\documentclass[12pt]{amsart}
\usepackage{amsmath,amsfonts,amsthm,amssymb,mathrsfs,amsxtra,amscd,latexsym}
\usepackage[hmargin=2cm,vmargin=3cm]{geometry}
\usepackage[dvipdfm, usenames]{color}
 \newtheorem{thm}{Theorem}[section]
 
 \newtheorem{cor}[thm]{Corollary}
 \newtheorem{lem}[thm]{Lemma}
 \newtheorem{prop}[thm]{Proposition}
 
 \newtheorem{dfn}[thm]{Definition}

 \newtheorem{rmk}[thm]{Remark}
 \newtheorem{exa}[thm]{Example}
 \theoremstyle{definition}
 \theoremstyle{remark}
 \numberwithin{equation}{section}

\newcommand{\sm}{\left(\begin{smallmatrix}}
\newcommand{\esm}{\end{smallmatrix}\right)}
\newcommand{\mat}{\left(\begin{matrix}}
\newcommand{\emat}{\end{matrix}\right)}

\newcommand{\mbf}{\mathbf}

\newcommand{\mc}{\mathcal}

\def\CC{\mathbb{C}}
\def\HH{\mathbb{H}}

\def\RR{\mathbb{R}}
\def\ZZ{\mathbb{Z}}

\def\m{\mathrm{mod}}

\def\GL{\mathrm{GL}}

\def\SL{\mathrm{SL}}

\begin{document}

\title{Structures for pairs of mock modular forms with the Zagier duality}


 \author{Dohoon Choi}
 \author{Subong Lim}

 \address{School of liberal arts and sciences, Korea Aerospace University, 200-1, Hwajeon-dong, Goyang, Gyeonggi 412-791, Republic of Korea}
  \email{choija@kau.ac.kr}

 \address{School of Mathematics, Korea Institute for Advanced Study, Hoegiro 85, Dongdaemun-gu, Seoul 130-722, Republic of Korea}
  \email{subong@kias.re.kr}

 \thanks{Keynote:  Zagier duality, Eichler integral, Supplementary function}

\begin{abstract}Zagier introduced special bases for weakly holomorphic modular forms to give the new proof of Borcherds' theorem on the infinite product expansions of integer weight modular forms on $\SL_2(\ZZ)$ with a Heegner divisor. These good bases appear in pairs, and they satisfy a striking duality, which is now called  Zagier duality. After the result of Zagier, this type duality was studied broadly in various view points including the theory of a mock modular form.  In this paper, we consider this problem with the Eichler cohomology theory, especially the supplementary function theory developed by Knopp. Using holomorphic Poincar\'e series and their supplementary functions, we construct a pair of families of vector-valued harmonic weak Maass forms satisfying the Zagier duality with integer weights $-k$ and $k+2$ respectively, $k>0$, for a $H$-group. We also investigate the structures of them such as the images under the  differential operators $D^{k+1}$ and $\xi_{-k}$ and quadric relations of the critical values of their $L$-functions.
\end{abstract}

\maketitle

\section{Introduction} \label{section1}

In his famous paper \cite{Zag}, Zagier displayed a beautiful result that the generating functions for traces of singular moduli are essentially weight $3/2$ weakly holomorphic modular forms. This result is related to the new proof of Borcherds' theorem on the infinite product expansions of integer weight modular forms on $\SL_2(\ZZ)$ with a Heegner divisor. Zagier's proof relies on special bases for weakly holomorphic modular forms which  appear in pairs and satisfy a striking duality, which is now called  $\it{Zagier\ duality}$.

Zagier gave a prominant example of this phenomenon in the spaces $M^!_{\frac 12, \chi}$ and $M^!_{\frac 32, \chi}$, where $M^!_{k,\chi}$ is the vector space of weight $k$ weakly holomorphic modular forms on $\Gamma_0(4)$ satisfying the Kohnen plus condition. Here,
$\chi$ is the $\theta$-multiplier system defined by $\chi(\gamma) = \frac{\theta(\gamma\tau)}{(c\tau+d)^{\frac12}\theta(\tau)}$ for $\gamma = \sm a&b\\c&d\esm \in \Gamma_0(4)$ and $\theta(\tau) = \sum_{n\in\ZZ} e^{2\pi in^2\tau}$.
There is a natural infinite basis $\{F_1(-1;\tau),\ F_1(-4;\tau),\ F_1(-5;\tau),\ \cdots\}$ for $M^!_{\frac 32, \chi}$, and the first few coefficients of these series are
\begin{eqnarray*}
F_1(-1;\tau) &=& q^{-1} -2  + 248q^{3} -492q^4 + 4119q^7 - \cdots,\\
F_1(-4;\tau) &=& q^{-4} -2 - 26752q^3- 143376q^4 - 8288256q^7-\cdots,\\
F_1(-5;\tau) &=& q^{-5} + 0 + 85995q^3 - 565760q^4 + 52756480q^7 - \cdots,
\end{eqnarray*}
where $q = e^{2\pi i\tau}$. Similarly, there is a basis $\{F_0(0;\tau),\ F_0(-3;\tau),\ F_0(-4;\tau), \cdots\}$ for $M^!_{\frac12, \chi}$, and the first few coefficients of these forms are
\begin{eqnarray*}
F_0(0;\tau) &=& 1 + 2q + 2q^4+ 0q^5 + \cdots,\\
F_0(-3;\tau) &=& q^{-3} -248q + 26752q^4 - 85995q^5 + \cdots,\\
F_0(-4;\tau) &=& q^{-4} + 492q + 143376q^4 + 565760q^5 + \cdots,\\
F_0(-7;\tau) &=& q^{-7} -4119q + 8288256q^4 - 52756480q^5 + \cdots.
\end{eqnarray*}
We can easily see that there is a striking pattern relating the coefficients of the $F_1(-m;\tau)$, grouped by column, and the coefficients of the individual forms $F_0(-n;\tau)$.

Although the existence of Zagier duality looks astonishing, many other examples were found recently.
Bringmann and Ono \cite{BO} established the Zagier duality of half-integer weight modular forms for $\Gamma_0(4)$ satisfying Kohnen plus condition. Folsom and Ono \cite{FO} found another example, which is especially interesting for its close connection to the Ramanujan mock theta function.

In the case of even integer weight $k$, the Zagier duality appeared in \cite{AKN} for small values of $k$ (namely, $k = 0,\ 4,\ 6,\ 8,\ 10,\ 14$). The results of \cite{AKN} were generalized in \cite{DJ} to arbitrary even $k$. Rouse \cite{Rou} also proved the Zagier duality between certain weakly holomorphic modular forms of weight $0$ and $2$ on $\Gamma_0(p)$ for $p \in \{5,\ 13,\ 17\}$ and the first author \cite{Cho} gave a simple proof of Rouse's result by using the residue theorem. Guerzhoy \cite{Gue} formulated this duality by saying that these coefficients constitute a grid and showed
the existence of the grid of weight $k\geq2$ for $\SL_2(\ZZ)$. Cho and Choie \cite{CC} derived the existence of the Zagier duality for vector-valued modular forms with the Weil representation for $\SL_2(\ZZ)$.
Recently, Bringmann, Kane and Rhoades \cite{BKR} discussed Zagier duality in a general context using the flipping operator, which is related to the theory of supplementary functions.

In this paper, we consider the grid of vector-valued modular forms for a $H$-group, from which the grid of (scalar-valued) modular forms follows as a corollary. Using supplementary functions, we construct a pair of families of vector-valued harmonic weak Maass forms $G_{n_2}(\tau)$ and vector-valued weakly holomorphic modular forms $f_{n_1}(\tau)$ satisfying the Zagier duality with integer weights $-k$ and $k+2$ respectively, $k>0$, for a $H$-group. Furthermore, we compute the images of vector-valued harmonic weak Maass forms in a grid under the two important operators $D^{k+1}$ and $\xi_{-k}$. Eventually we can interpret the structure of forms $f_n(\tau)$ and $G_n(\tau)$ satisfying the Zagier duality in term of holomorphic Poincar\'e series:
$$
\begin{CD}
f_n=P_{n} @<~~~~~supplementary~~~<< P_{-n}\\
@AAD^{k+1}A @AA \xi_{-k} A\\
G_n @= G_n
\end{CD}
$$
Here, $P_{-n}(\tau)$ denotes a holomorphic Poincar\'e series of order $n, n>0$. These results give another symmetries: quadric relations of the critical values of their $L$-functions.

The supplementary function theory was developed by Knopp \cite{Kno} for scalar-valued modular forms of integer weights on a $H$-group and Gim\'enez \cite{Gim} considered the supplementary function theory for vector-valued modular forms of integer weights on $\SL_2(\ZZ)$. Based on the arguments of Knopp in \cite{Kno}, for our purpose we extend this theory to vector-valued modular forms of integer weights on a $H$-group.

First we recall some definitions and fix notations.
Let $\Gamma$ be a $H$-$\it{group}$, i.e., a finitely generated Fuchsian group of the first kind which has at least one parabolic class. Let $k\in\ZZ$ and $\chi$ a (unitary) character.  Let $p$ be a positive integer and $\rho:\Gamma\to \GL(p,\CC)$ a $p$-dimensional unitary complex representation.  We suppose that $\rho(B)$ is diagonal for all parabolic elements $B\in\Gamma$ and $\rho(-I_2) = I_p$, where $I_j$ is the identity matrix of order $j\in\ZZ_{>0}$. Let $T = \sm 1&\lambda\\ 0&1\esm,\ \lambda>0$, a generator of $\Gamma_{\infty}$, where $\Gamma_\infty$ denotes the stabilizer of $\infty$. Then
\[\chi(T)\rho(T) = \sm e^{2\pi i\kappa_1}& & & & \\
						                 & \cdot &&&\\
						                 && \cdot &&\\
						                 &&& \cdot &\\
						                 &&&& e^{2\pi i\kappa_p}\esm,\]
where $0\leq \kappa_j<1$ for $1\leq j\leq p$.  Also, if we let
\[\bar{\chi}(T)\bar{\rho}(T) = \sm e^{2\pi i\kappa'_1}& & & & \\
						                 & \cdot &&&\\
						                 && \cdot &&\\
						                 &&& \cdot &\\
						                 &&&& e^{2\pi i\kappa'_p}\esm,\]
with $0\leq\kappa_j'<1$ for $1\leq j\leq p$, then
\begin{equation*}
\kappa_j' =
\begin{cases}
0 & \text{if $\kappa_j=0$},\\
1-\kappa & \text{if $\kappa_j>0$},
\end{cases}
\end{equation*}
for $1\leq j\leq p$. We denote the standard basis elements of the vector space $\CC^p$ by $\mbf{e}_j$ for $1\leq j\leq p$.

We write $M^!_{k+2,\chi,\rho}(\Gamma)$ (resp. $H_{k+2,\chi,\rho}(\Gamma)$) for the space of vector-valued weakly holomorphic modular forms (resp. vector-valued harmonic weak Maass forms) of weight $k+2$, character $\chi$ and type $\rho$ on $\Gamma$ (see section \ref{section2}).
Two differential operators $\xi_{-k} := 2iv^{-k}(\overline{\frac{\partial }{\partial\bar{\tau}}})$ and $D^{k+1} := \biggl(\frac{1}{2\pi i}\frac{d}{d\tau}\biggr)^{k+1}$ play a central role in the theory of harmonic weak Maass forms, where $\tau = u+iv\in\HH$. The assignment $F(\tau)\mapsto \xi_{-k}(F)(\tau)$ gives an anti-linear mapping
\[\xi_{-k}: H_{-k,\chi,\rho}(\Gamma) \to M^!_{k+2,\bar{\chi},\bar{\rho}}(\Gamma)\]
and the assignment $F(\tau)\mapsto D^{k+1}(F)(\tau)$ gives a linear mapping
\[D^{k+1}: H_{-k,\chi,\rho}(\Gamma) \to M^!_{k+2,\chi,\rho}(\Gamma).\]
Let $H^*_{-k,\chi.\rho}(\Gamma)$ be the inverse image of the space of vector-valued cusp forms $S_{k+2,\bar{\chi},\bar{\rho}}(\Gamma)$ under the mapping $\xi_{-k}$. Any harmonic weak Maass form $F(\tau)\in H^*_{-k,\chi,\rho}(\Gamma)$ has a unique decomposition $F(\tau) = F^+(\tau) + F^-(\tau)$, where
\begin{eqnarray*}
F^+(\tau) &=& \sum_{j=1}^p \sum_{n\gg-\infty} a^+(n,j) e^{2\pi i(n+\kappa_j)\tau/\lambda} \mbf{e}_j,\\
F^-(\tau) &=& \sum_{j=1}^p \sum_{n+\kappa_j<0} a^-(n,j) H(2\pi (n+\kappa_j)v/\lambda)e^{2\pi i(n+\kappa_j)u/\lambda}\mbf{e}_j,
\end{eqnarray*}
where $H(w) = e^{-w}\int^\infty_{-2w}e^{-t}t^{k}dt$ (for this decomposition, see \cite[section 3]{BF}).
The first (resp. second) summand is called the {\it{holomorphic}} (resp. {\it{non-holomorphic}}) part of $F(\tau)$.

We use the term $\it{grid}$ to explain the Zagier duality more systematically, following \cite{Gue}.
\begin{dfn} We call two collections $f_{n_1,\alpha_1,\chi,\rho}(\tau)\in M^!_{k+2,\chi,\rho}(\Gamma)$ and $G_{n_2,\alpha_2,\bar{\chi},\bar{\rho}}(\tau)\in H^*_{-k,\bar{\chi},\bar{\rho}}(\Gamma)$ a vector-valued grid of weight $k+2$, character $\chi$ and type $\rho$ on $\Gamma$ if the following conditions hold for every indices $n_1$ and $n_2$ satisfying $n_1-\kappa_{\alpha_1}\geq0$ and $n_2-\kappa'_{\alpha_2}>0$:
\begin{enumerate}
\item[(1)]  $f_{n_1,\alpha_1,\chi,\rho}(\tau)$ and $G^+_{n_2,\alpha_2,\bar{\chi},\bar{\rho}}(\tau)$ vanish at all cusps of $\Gamma$ which are not equivalent to $i\infty$.
\item[(2)] At  $i\infty$, $f_{n_1,\alpha_1,\chi,\rho}(\tau)$ and $G^+_{n_2,\alpha_2,\bar{\chi},\bar{\rho}}(\tau)$ have expansions of the forms
\begin{eqnarray*}
f_{n_1,\alpha_1,\chi,\rho}(\tau) &=& e^{2\pi i(-n_1+\kappa_{\alpha_1})\tau/\lambda}\mbf{e}_{\alpha_1} + \sum_{j=1}^p \sum_{l+\kappa_j>0} a_{n_1,\alpha_1,\chi,\rho}(l,j) e^{2\pi i(l+\kappa_j)z/\lambda}\mbf{e}_j,\\
G^+_{n_2,\alpha_2,\bar{\chi},\bar{\rho}}(\tau) &=& e^{2\pi i(-n_2+\kappa'_{\alpha_2})\tau/\lambda}\mbf{e}_{\alpha_2} + \sum_{j=1}^p \sum_{l+\kappa'_j\geq0} b_{n_2,\alpha_2,\bar{\chi},\bar{\rho}}(l,j) e^{2\pi i(l+\kappa'_j)z/\lambda}\mbf{e}_j.
\end{eqnarray*}
\item[(3)] Fourier coefficients of $f_{n_1,\alpha_1,\chi,\rho}(\tau)$ and $G^+_{n_2,\alpha_2,\bar{\chi},\bar{\rho}}(\tau)$ satisfy the identity
\begin{equation*}
a_{n_1,\alpha_1,\chi,\rho}(n_2-(\kappa_{\alpha_2}+\kappa'_{\alpha_2}),\alpha_2) = - b_{n_2,\alpha_2,\bar{\chi},\bar{\rho}}(n_1-(\kappa_{\alpha_1}+\kappa'_{\alpha_1}),\alpha_1).
\end{equation*}
\end{enumerate}
\end{dfn}

In this paper we consider the grid of weight $k\in\ZZ$ on a $H$-group $\Gamma$ for vector-valued modular forms. We  assume that $-I_2\in\Gamma$.
Let $P_{n_1,\alpha_1,\chi,\rho}(\tau)$ be a Poincar\'e series of order $-n_1$ defined by
\begin{equation} \label{dfnofpoincareseries}
P_{n_1,\alpha_1,\chi,\rho}(\tau) := \frac12\sum_{\gamma} \frac{e^{2\pi i(-n_1+\kappa_{\alpha_1})\gamma \tau/\lambda}}{\chi(\gamma)(c\tau+d)^{k+2}}\rho(\gamma)^{-1}\mbf{e}_{\alpha_1},
\end{equation}
where $\gamma= \sm *&*\\c&d\esm$ runs through a complete set of elements of $\Gamma$ with distinct lower row. In the following theorem we give a description of a grid in terms of these Poincar\'e series.

\begin{thm} \label{main1} Let $\Gamma$ be a $H$-group with $-I_2\in \Gamma$.
Suppose that $k$ is a positive integer, $\chi$ is a character and $\rho$ is a unitary representation such that $\rho(B)$ is diagonal for all parabolic elements $B\in\Gamma$
 and $\rho(-I_2) = I_p$.
Then there exists a unique vector-valued grid of weight $k+2$, character $\chi$ and type $\rho$. Moreover, we have
\[f_{n_1,\alpha_1,\chi,\rho}(\tau) = P_{n_1,\alpha_1,\chi,\rho}(\tau).\]
We also have
\begin{eqnarray*}
(D^{k+1} G_{n_2,\alpha_2,\bar{\chi},\bar{\rho}})(\tau) &=&  \biggl( \frac{-n_2+\kappa'_{\alpha_2}}{\lambda}\biggr)^{k+1}P_{n_2,\alpha_2,\bar{\chi},\bar{\rho}}(\tau)
\end{eqnarray*}
and
\[(\xi_{-k}G_{n_2,\alpha_2,\bar{\chi},\bar{\rho}})(\tau) = \frac{(-4\pi)^{k+1}}{\Gamma(k+1)}\biggl( \frac{-n_2+\kappa'_{\alpha_2}}{\lambda}\biggr)^{k+1}P_{n_2',\alpha_2,\chi,\rho}(\tau),\]
where
\begin{equation*}
n_2' =
\begin{cases}
-n_2 & \text{if $\kappa_{\alpha_2}=0$},\\
1-n_2 & \text{if $\kappa_{\alpha_2}>0$}.
\end{cases}
\end{equation*}
\end{thm}

\begin{rmk} We remark the following.
\begin{enumerate}
\item In particular,  if we let $p=1$ and $\rho$ a trivial representation, then we 
    have only one component.    
      Therefore,  we  obtain the Zagier duality for (scalar-valued) modular forms. In this case, we omit the $\rho$ and $\alpha$ in the notation for convenience.

\item In Theorem \ref{main1}, we showed that there is a vector-valued grid of integer weight if $\rho$ satisfies $\rho(-I_2) = I_p$, but, the Weil representation does not satisfy this condition. Actually, if $\rho$ is the Weil representation, then $\rho(-I_2)$ interchanges $e_{\alpha}$ and $e_{-\alpha}$. In this case, we can  see that we get the same result as in Theorem \ref{main1} with a modified definition of a grid: $f_{n_1,\alpha_1,\chi,\rho}(\tau)\in M^!_{k+2,\chi,\rho}(\Gamma)$ and $G^+_{n_2,\alpha_2,\bar{\chi},\bar{\rho}}(\tau)$ have expansions of the form at $i\infty$
\begin{eqnarray*}
f_{n_1,\alpha_1,\chi,\rho}(\tau) &=& e^{2\pi i(-n_1+\kappa_{\alpha_1})\tau/\lambda}\mbf{e}_{\alpha_1} + e^{2\pi i(-n_1+\kappa_{\alpha_1})\tau/\lambda}\mbf{e}_{-\alpha_1}\\
&&+ \sum_{j=1}^p \sum_{l+\kappa_j>0} a_{n_1,\alpha_1}(l,j) e^{2\pi i(l+\kappa_j)z/\lambda}\mbf{e}_j,
\end{eqnarray*}
and
\begin{eqnarray*}
G^+_{n_2,\alpha_2,\bar{\chi},\bar{\rho}}(\tau) &=& e^{2\pi i(-n_2+\kappa'_{\alpha_2})\tau/\lambda}\mbf{e}_{\alpha_2} +e^{2\pi i(-n_2+\kappa'_{\alpha_2})\tau/\lambda}\mbf{e}_{-\alpha_2}\\
&& + \sum_{j=1}^p \sum_{l+\kappa'_j\geq0} b_{n_2,\alpha_2}(l,j) e^{2\pi i(l+\kappa'_j)z/\lambda}\mbf{e}_j.
\end{eqnarray*}

\item  In \cite{Zag2}, Zagier defined mock modular forms and
the definition of Zagier implies that holomorphic parts of harmonic weak Maass forms are mock modular forms. So we can also define a grid in terms of mock modular forms. For example, if $\chi$ is trivial, then  $G^+_{n_2,\alpha_2,\bar{\chi},\bar{\rho}}(\tau)$ is a mock modular form such that its shadow is the Poincar\'e series of order $n_2$, $P_{-n_2,\alpha_2,\chi,\rho}(\tau)$, and $D^{k+1}(G^+_{n_2,\alpha_2,\bar{\chi},\bar{\rho}})(\tau)$ is the Poincar\'e series of order $-n_2$, $P_{n_2,\alpha_2,\bar{\chi},\bar{\rho}}(\tau)$.
\end{enumerate}
\end{rmk}

Now we know that the Zagier duality is closely related to the theory of mock modular forms. Thus it is natural to consider the symmetries coming from their shadows.
From Theorem \ref{main1} the shadows come from the Poincar\'e series $P_{n_2', \alpha_2,\chi,\rho}(\tau)$ with the same weight.
Thus the following corollary immediately follows.

\begin{cor} \label{main2}
Let $\Gamma, k, \chi$ and $\rho$ be given as in Theorem \ref{main1}.
Suppose that
\begin{eqnarray*}
G_{n_2,\alpha_2,\bar{\chi},\bar{\rho}}^-(\tau) &=& \sum_{j=1}^p \sum_{l+\kappa'_j<0} b^-_{n_2,\alpha_2,\bar{\chi},\bar{\rho}}(l,j)H(2\pi (l+\kappa_j')v/\lambda)e^{2\pi i(l+\kappa'_j)u/\lambda}\mbf{e}_j
\end{eqnarray*}
and
\begin{eqnarray*}
G_{\tilde{n}_2,\tilde{\alpha}_2,\bar{\chi},\bar{\rho}}^-(\tau) &=& \sum_{j=1}^p \sum_{l+\kappa'_j<0} \tilde{b}^-_{\tilde{n}_2,\tilde{\alpha}_2,\bar{\chi},\bar{\rho}}(l,j)H(2\pi (l+\kappa_j')v/\lambda)e^{2\pi i(l+\kappa'_j)u/\lambda}\mbf{e}_j
\end{eqnarray*}
are two non-holomorphic parts of harmonic weak Maass forms in a grid. Then we have
\[b^-_{n_2,\alpha_2,\bar{\chi},\bar{\rho}}(-\tilde{n}_2,\tilde{\alpha}_2)(-\tilde{n}_2+\kappa'_{\tilde{\alpha}_2})^{k+1} = \overline{b^-_{\tilde{n}_2,\tilde{\alpha}_2,\bar{\chi},\bar{\rho}}(-n_2,\alpha_2)}(-n_2+\kappa'_{\alpha_2})^{k+1}\]
for every indices $n_2$ and $\tilde{n}_2$ satisfying $n_2-\kappa'_{\alpha_2}>0$ and $\tilde{n}_2-\kappa'_{\tilde{\alpha}_2}>0$.
\end{cor}



The study of $L$-functions has been a central theme in number theory since the pioneering work of Riemann when the first instances of a prime number theorem were formulated.
Especially, the theory of modular forms is very closely related to $L$-functions. For example, the special values of $L$-functions associated to modular forms are important objects in the conjectures posed by Birch-Swinnerton-Dyer, Beilinson, and Bloch-Kato.
     Recently, Bringmann, Fricke and Kent \cite{BFK} found the definition of $L$-functions  for weakly holomorphic modular forms  by resolving the difficulties due to the issues coming from their bad asymptotic behaviors
    at cusps.
 We denote by $L(f,\zeta_{c\lambda}^{-d},s)$ the twisted $L$-function of a weakly holomorphic modular form $f(\tau)$  (for the precise definition, see section \ref{section3}).
 In the following theorem,  with the definition of $L$-functions associated to weakly holomorphic modular forms,
 we prove  that the critical values of twisted $L$-functions of mock modular forms appearing in a grid satisfy a quadric relation and the Fourier coefficients of their shadows can be written in terms of these critical values.

\begin{thm} \label{main3}
Assume that $\Gamma, k$ and $\chi$ are given as in Theorem \ref{main1}.
Let $\gamma_1 = \sm a_1&b_1\\c_1&d_1\esm ,\cdots, \gamma_t= \sm a_t&b_t\\c_t&d_t\esm$ be generators of $\Gamma$.
Suppose that $P_{n_i,\chi}(\tau)$  is a Poincar\'e series which is a cusp form, defined as
in (\ref{dfnofpoincareseries}).
We write its Fourier expansion as
\[P_{n_i,\chi}(\tau) = \sum_{l+\kappa>0} c_{n_i,\chi}(l)e^{2\pi i(l+\kappa)\tau/\lambda},\]
for $i=1,\ 2$.
Then there exist   complex constants $A_{\alpha,\beta}(i,j),\ B_{\alpha}(i),\ C_{\beta}(j),\ D,\ 0\leq \alpha,\beta\leq k, 1\leq i,j\leq t$, such that
\begin{eqnarray*}
(-n_2+\kappa)^{-(k+1)}c_{n_1,\chi}(-n_2)&=& \sum_{1\leq i,j\leq t}\sum_{0\leq\alpha,\beta\leq k}A_{\alpha,\beta}(i,j)\overline{L(f_{n_1',\bar{\chi}},\zeta_{c_j\lambda}^{-d_j},\alpha+1)}L(f_{n_2', \bar{\chi}},\zeta_{c_j\lambda}^{-d_j},\beta+1)\\
&&+\delta_{\kappa,0}n_2^{-(k+1)}b_{-n_2,\bar{\chi}}(0)\sum_{i=1}^t\sum_{\alpha=0}^k B_\alpha(i)
\overline{L(f_{n_1', \bar{\chi}},\zeta_{c_i\lambda}^{-d_i},\alpha+1)}\\
&&+\delta_{\kappa,0}n_1^{-(k+1)}\overline{b_{-n_1,\bar{\chi}}(0)}\sum_{j=1}^t\sum_{\beta=0}^k  C_\beta(j)  L(f_{n_2', \bar{\chi}},\zeta_{c_j\lambda}^{-d_j},\beta+1)\\
&&+ \delta_{\kappa,0}D(n_1n_2)^{-(k+1)}\overline{b_{-n_1,\bar{\chi}}(0)}b_{-n_2,\bar{\chi}}(0)
\end{eqnarray*}
for every integer $n_1,n_2$ such that $-n_1+\kappa, -n_2+\kappa>0$. Here, $b_{-n,\bar{\chi}}(0)$ is the $0$-th Fourier coefficient of $G^+_{-n,\bar{\chi}}(\tau)$. 
\end{thm}

\begin{exa} Let $\Gamma = \Gamma_0(N)$ for some positive integer $N$ and let $\chi$ be a character given by $\eta$ function, i.e.,
\[\chi(\gamma) = \biggl(\frac{\eta(\gamma\tau)}{(c\tau+d)^{\frac12}\eta(\tau)}\biggr)^{2k}\]
for $\gamma = \sm a&b\\c&d\esm \in \SL_2(\ZZ)$,
where $\eta(\tau) = e^{\pi i\tau/12}\prod_{n=1}^\infty (1-e^{2\pi in\tau})$. In this case, we see that $T = \sm 1&1\\0&1\esm$ and $\chi(T) = e^{\pi ik/6}$. If $k\not\equiv 0\ (\m\ 12)$, then $\kappa \neq0$ for $\chi$. Then the formula in Theorem \ref{main3} reduces to
\begin{eqnarray*}
(-n_2+\kappa)^{(k+1)}c_{n_1,\chi}(-n_2)
= \sum_{1\leq i,j\leq t}\sum_{0\leq\alpha,\beta\leq k}A_{\alpha,\beta}(i,j)\overline{L(f_{n_1',\bar{\chi}},\zeta_{c_j\lambda}^{-d_j},\alpha+1)}L(f_{n_2', \bar{\chi}},\zeta_{c_j\lambda}^{-d_j},\beta+1).
\end{eqnarray*}
\end{exa}


The remainder of this paper is organized as follows. In section \ref{section2}, we  introduce the basic notions of vector-valued modular forms, derive the Fourier expansions of the Poincar\'e series and review the supplementary function theory. In section \ref{section3}, we describe $L$-functions of weakly holomorphic cusp forms and explain their relation to period polynomials. In section \ref{section4}, we prove the main results: Theorem \ref{main1}, Corollary \ref{main2},  and Theorem \ref{main3}.

\section{Vector-valued modular forms} \label{section2}

In this section, we introduce vector-valued modular forms and vector-valued harmonic weak Maass forms. Also, we recall the definitions and properties of two differential operators $\xi_{-k}$ and $D^{k+1}$, which are important to study vector-valued harmonic weak Maass forms. Moreover, we derive  the Fourier coefficients of vector-valued Poincar\'e series and recall the supplementary function theory.

\subsection{Vector-valued modular forms} \label{section2.1}
 We begin by introducing the definition of the vector-valued modular forms. Let $\Gamma$ be a $H$-$\it{group}$, i.e., a finitely generated Fuchsian group of the first kind which has at least one parabolic class. Let $k\in\ZZ$ and $\chi$ a (unitary) character. Thus $\chi(\gamma)$ is a complex number independent of $\tau$ such that
\begin{enumerate}
\item $|\chi(\gamma)| = 1$ for all $\gamma\in\Gamma$.
\item $\chi$ satisfies the non-triviality condition:
$\chi(-I) = e^{\pi ik}$.
\end{enumerate}
Let $p$ be a positive integer and $\rho:\Gamma\to \GL(p,\CC)$ a $p$-dimensional unitary complex representation.  We suppose that $\rho(B)$ is diagonal for all parabolic elements $B\in\Gamma$
 and $\rho(-I_2) = I_p$, where $I_j$ is the identity matrix of order $j\in\ZZ_{>0}$.  We denote the standard basis elements of the vector space $\CC^p$ by $\mbf{e}_j$ for $1\leq j\leq p$.

\begin{dfn} \label{dfnofvvm} A vector-valued weakly holomorphic modular form of weight $k$, multiplier system $\chi$ and type $\rho$ on $\Gamma$ is a sum $f(\tau) = \sum_{j=1}^p f_j(\tau)\mbf{e}_j$ of functions holomorphic in the complex upper half-plane $\HH$ satisfying the following conditions:
\begin{enumerate}
\item For all $\gamma \in\Gamma$, we have
$(f|_{k,\chi,\rho} \gamma)(\tau) = f(\tau)$.

\item For each parabolic cusp $q$, there is $\gamma_q = \sm a&b\\c&d\esm \in \SL_2(\RR)$ such that $\gamma_q\infty = q$. Then
each function $(c\tau+d)^{-k}f(\gamma_q\tau)$ has a Fourier expansion of the form
\[(c\tau+d)^{-k}f(\gamma_q\tau) = \sum_{j=1}^p\sum_{n\gg-\infty} a_{j,q}(n)e^{2\pi i(n+\kappa_{j,q})\tau/\lambda_{q}}\mbf{e}_j,\]
where $\kappa_{j,q}$ (resp. $\lambda_{q}$) is a constant which depends on $j$ and $q$ (resp. $q$).
\end{enumerate}
\end{dfn}

Here, the slash operator $|_{k,\chi,\rho} \gamma$ is defined by
\[(f|_{k,\chi,\rho}\gamma)(\tau) = \chi(\gamma)^{-1}(c\tau+d)^{-k}\rho^{-1}(\gamma)f(\gamma\tau)\]
for $\gamma = \sm a&b\\c&d\esm \in \Gamma$, where $\gamma\tau = \frac{a\tau+b}{c\tau+d}$.
The space of all vector-valued weakly holomorphic modular forms  of weight $k$, character $\chi$ and type $\rho$ on $\Gamma$ is denoted by $M^!_{k,\chi,\rho}(\Gamma)$. There are subspaces $M_{k,\chi,\rho}(\Gamma)$ and $S_{k,\chi,\rho}(\Gamma)$ of {\it{vector-valued\ holomorphic\ modular\ forms}} and {\it{vector-valued cusp forms}}, respectively, for which we require that each $a_{j,q}(n) = 0$ when $n+\kappa_{j,q}$ is negative, respectively, non-positive.
For a Fourier expansion $\sum_{j=1}^p\sum_{n\gg-\infty}a(n,j)e^{2\pi i(n+\kappa_j)\tau/\lambda}\mbf{e}_j$, the principal part is
 \[\sum_{j=1}^p\sum_{n+\kappa_j<0}a(n,j)e^{2\pi i(n+\kappa_j)\tau/\lambda}\mbf{e}_j\]
and the constant term is
\[\sum_{j=1}^p\sum_{n+\kappa_j=0}a(n,j)e^{2\pi i(n+\kappa_j)\tau/\lambda}\mbf{e}_j.\]

Now we give the definition of vector-valued harmonic weak Maass forms and related operators. To define vector-valued harmonic weak Maass forms, we need to introduce the weight $k$ hyperbolic Laplacian given by
\[\Delta_k := -v^2\biggl(\frac{\partial^2}{\partial u^2}+\frac{\partial^2}{\partial v^2}\biggr) + ikv\biggl(\frac{\partial}{\partial u}+i\frac{\partial}{\partial v}\biggr).\]

\begin{dfn} \label{dfnofhar} A vector-valued harmonic weak Maass form of weight $k$, character $\chi$ and type $\rho$ on $\Gamma$ is a sum $f(\tau) = \sum_{j=1}^p f_j(\tau)\mbf{e}_j$ of smooth functions on $\HH$ satisfying
\begin{enumerate}
\item[(1)] $(f|_{k,\chi,\rho}\gamma)(\tau) = f(\tau)$ for all $\gamma\in\Gamma$,
\item[(2)] $\Delta_k f =0$,
\item[(3)] there is a constant $C>0$ such that
\[(c\tau+d)^{-k}f_j(\gamma_q\tau) = O(e^{Cv})\]
as $v\to\infty$ $($uniformly in $u)$ for every integer $1\leq j\leq p$ and every element $\gamma_q=\sm a&b\\c&d\esm\in \SL_2(\RR)$ as in (2) of Definition \ref{dfnofvvm}, where $\tau = u+iv\in\HH$.
\end{enumerate}
We write $H_{k,\chi,\rho}(\Gamma)$ for the space of vector-valued harmonic weak Maass forms of weight $k$, character $\chi$ and type $\rho$ on $\Gamma$.
\end{dfn}

 Let $T = \sm 1&\lambda\\ 0&1\esm,\ \lambda>0$, be a generator of $\Gamma_{\infty}$, where $\Gamma_\infty$ is the stabilizer of $\infty$. Then we have
\begin{equation} \label{kappa}
\chi(T)\rho(T) = \sm e^{2\pi i\kappa_1}& & & & \\
						                 & \cdot &&&\\
						                 && \cdot &&\\
						                 &&& \cdot &\\
						                 &&&& e^{2\pi i\kappa_p}\esm,
\end{equation}
where $0\leq \kappa_j<1$ for $1\leq j\leq p$.
Here we choose the function
\[H(w) = e^{-w}\int^\infty_{-2w}e^{-t}t^{-k}dt.\]
The integral converges for $k<1$ and can be holomorphically continued in $k$ (for $w\neq0$) in the same way as the Gamma function. If $w<0$, then $H(w) = e^{-w}\Gamma(1-k, -2w)$, where $\Gamma(a,x)$ denotes the incomplete Gamma function as in \cite{Abr}.
Any harmonic weak Maass form $f(\tau)$ of weight $k$ has a unique decomposition $f(\tau) = f^+(\tau) + f^-(\tau)$, where
\begin{eqnarray} \label{decomposition}
f^+(\tau) &=& \sum_{j=1}^p\sum_{n\gg-\infty} a^+(n,j) e^{2\pi i(n+\kappa_j)\tau/\lambda}\mbf{e}_j,\\
\nonumber f^-(\tau) &=& \sum_{j=1}^p \biggl(\delta_{\kappa_j,0}a^-(0,j)v^{1-k} + \sum_{n\ll \infty\atop n+\kappa_j\neq 0} a^-(n,j) H(2\pi (n+\kappa_j)v/\lambda)e^{2\pi i(n+\kappa_j)u/\lambda}\biggr)\mbf{e}_j,
\end{eqnarray}
where $\delta_{\kappa_j,0} = 1$ if $\kappa_j=0$ or $0$ if $\kappa_j > 0$ (for the decomposition, see \cite[section 3]{BF}).
The first (resp. second) summand is called the $\it{holomorphic}$ (resp. {\it{non-holomorphic}}) part of $f(\tau)$.

Let us  recall the Maass raising and lowering operators on non-holomorphic modular forms of weight $k$. They are defined as the differential operators
\[R_k = 2i\frac{\partial}{\partial \tau} + kv^{-1}\ \text{and}\ L_k = -2iv^2\frac{\partial}{\partial \bar{\tau}}.\]

\begin{thm} \cite[Proposition 3.2]{BF} \label{operator}
 Let $f(\tau) \in H_{-k,\chi,\rho}(\Gamma)$.
  The assignment $f(\tau) \mapsto \xi_{-k}(f)(\tau) := v^{-k-2}\overline{L_{-k} f(z)} = R_{k}v^{-k}\overline{f(\tau)}$ defines an anti-linear mapping
\[\xi_{-k} : H_{-k,\chi,\rho}(\Gamma) \to M^!_{k+2,\bar{\chi},\bar{\rho}}(\Gamma).\]
Its kernel is $M^!_{-k,\chi,\rho}(\Gamma)$.
\end{thm}

We let $H^*_{-k,\chi,\rho}(\Gamma)$ denote the inverse image of the space of cusp forms $S_{k+2,\bar{\chi},\bar{\rho}}(\Gamma)$ under the mapping $\xi_{-k}$. Hence, if $f(\tau)\in H^*_{-k,\chi,\rho}(\Gamma)$, then the Fourier coefficients $a^-(n,j)$ as in (\ref{decomposition})  vanish if $n+\kappa_j$ is non-negative.
Now we introduce another differential operator $D^{k+1}$ on $H^*_{-k,\chi,\rho}(\Gamma)$ for $k>0$ as
\[D^{k+1} = \biggl( \frac{1}{2\pi i}\frac{d}{d\tau}\biggr)^{k+1}.\]
Let $f(\tau)\in H^*_{-k,\chi,\rho}(\Gamma)$ whose Fourier expansion is given as in (\ref{decomposition}). Then we have (see \cite{BOR})
\[(D^{k+1}f)(\tau) = \sum_{j=1}^p \sum_{n\gg-\infty} a^+(n,j)\biggl(\frac{n+\kappa_j}\lambda\biggr)^{k+1}e^{2\pi i(n+\kappa_j)\tau/\lambda}\mbf{e}_j.\]





\subsection{Vector-valued Poincar\'e series} \label{section2.2}
In this subsection we begin to develop a theory of Poincar\'e series for vector-valued modular forms that runs parallel to the classical theory \cite{Leh, Ran}. Vector-valued Poincar\'e series was defined and its Fourier coefficients were given in the case of $\Gamma = \SL_2(\ZZ)$ and real $k>2$ in \cite{KM}. We review the result of Knopp and Mason in \cite{KM}  and extend their result to arbitrary $H$-group.

\begin{dfn} Fix integer $n$ and $\alpha$ with $1\leq \alpha\leq p$. The Poincar\'e series $P_{n,\alpha,\chi,\rho}(\tau)$ is defined as
\begin{equation} \label{poincare}
P_{n,\alpha,\chi,\rho}(\tau) := \frac12\sum_{\gamma=\sm a&b\\c&d\esm} \frac{e^{2\pi i(-n+\kappa_\alpha)\gamma\tau/\lambda}}{\chi(\gamma)(c\tau+d)^k}\rho(\gamma)^{-1}\mbf{e}_\alpha,
\end{equation}
where $\gamma = \sm a&b\\c&d\esm$ ranges over a set of coset representatives for $<T>\setminus \Gamma$ and $\kappa_\alpha$ is as in (\ref{kappa}).
\end{dfn}

The series (\ref{poincare}) is well-defined and     invariant with respect to the action $|_{k,\chi,\rho}$ of $\Gamma$ if we assume absolute convergence.

\begin{prop} \label{convergence}  If $k>2$, then the component function $(P_{n,\alpha,\chi,\rho})_j(\tau)$ of $P_{n,\alpha,\chi,\rho}(\tau),\ 1\leq j\leq p$, converges absolutely uniformly on compact subsets of $\HH$. In particular, each $(P_{n,\alpha,\chi,\rho})_j(\tau)$ is holomorphic in $\HH$.
\end{prop}

\begin{proof} [\bf Proof of Proposition \ref{convergence}]
We can  see that
\[(P_{n,\alpha,\chi,\rho})_j(\tau) = \frac12\sum_{\gamma=\sm a&b\\c&d\esm}e^{2\pi i(-n+\kappa_\alpha)\gamma\tau/\lambda}\chi(\gamma)^{-1}(c\tau+d)^{-k}\rho(\gamma^{-1})_{j,\alpha},\]
and we have to estimate the exponential term. Setting $\tau = u+iv$, we see that
\[\biggl|e^{2\pi i(-n+\kappa_\alpha)\gamma\tau/\lambda}\biggr| \leq e^{2\pi(|n|+1)/(c_0v)},\]
for $c\neq0$, where $c_0 = \min \{c>0|\ \exists \sm a&b\\c&d\esm\in\Gamma\}$. Note that $c=0$ corresponds to the term $\gamma = \pm I_2$. Now a standard argument shows that if $k>2$ then the function $(P_{n,\alpha,\chi,\rho})_j(\tau)$ converges absolutely uniformly on compact subsets of $\HH$.
\end{proof}

It remains for us to consider the Fourier expansions of the component functions $(P_{n,\alpha,\chi,\rho})_j(\tau)$. The explicit expressions that arise are familiar to those who have studied the Fourier coefficients of (scalar-valued) modular forms.   We begin by introducing some of the notation:
\begin{eqnarray*}
C &:=& \{\sm a&b\\c&d\esm\in\Gamma|\ 0\leq -d,\ a < |c|\lambda\},\\
C^+ &:=& \{\sm a&b\\c&d\esm\in C|\ c>0\}.
\end{eqnarray*}
Then  $C$ is the set of double coset representatives for $<T>\setminus\Gamma/<T>$.
We also use the Bessel functions of the first kind \cite{Wat}:
\begin{eqnarray*}
J_n(z) &=& \sum_{t=0}^\infty \frac{(-1)^t(z/2)^{n+2t}}{t!\Gamma(n+t+1)},\\
I_n(z) &=& \sum_{t=0}^\infty \frac{(z/2)^{n+2t}}{t!\Gamma(n+t+1)}.
\end{eqnarray*}

\begin{thm} \label{fourierpoincare} Let $P_{n,\alpha,\chi,\rho}(\tau)$ be the Poincar\'e series in (\ref{poincare}). As long as $k>2$, $P_{n,\alpha,\chi,\rho}(\tau)\in M^!_{k,\chi,\rho}(\Gamma)$ and has a Fourier expansion of the form at $i\infty$
\[(P_{n,\alpha,\chi,\rho})_j(\tau) = \delta_{j,\alpha}e^{2\pi i(-n+\kappa_\alpha)\tau/\lambda} + \sum_{l+\kappa_j>0} a_{n,\alpha,\chi,\rho}(l,j)e^{2\pi i(l+\kappa_j)\tau/\lambda},\]
where the Fourier coefficients $a_{n,\alpha,\chi,\rho}(l,j)$ are as follows and exactly one of the following holds:
\begin{itemize}
\item[(1)] If $-n+\kappa_\alpha>0$, then
\begin{eqnarray*}
a_{n,\alpha,\chi,\rho}(l,j) &=& \frac{2\pi i^{-k}}{\lambda}\sum_{\gamma=\sm a&b\\c&d\esm\in C^+}c^{-1}\biggl(\frac{l+\kappa_j}{-n+\kappa_\alpha}\biggr)^{\frac{k-1}2}\chi(\gamma)^{-1}\rho(\gamma^{-1})_{j,\alpha}e^{\frac{2\pi i}{c\lambda}((-n+\kappa_\alpha)a+(l+\kappa_j)d)}\\
&&\times J_{k-1}\biggl(\frac{4\pi}{c\lambda}\sqrt{(-n+\kappa_\alpha)(l+\kappa_j)}\biggr)
\end{eqnarray*}

\item[(2)] If $-n+\kappa_\alpha = 0$, then
\[a_{n,\alpha,\chi,\rho}(l,j) = \frac{(-2\pi i)^{k}}{\Gamma(k)\lambda^k}\sum_{\gamma=\sm a&b\\c&d\esm\in C^+}c^{-k}(l+\kappa_j)^{k-1}\chi(\gamma)^{-1}\rho(\gamma^{-1})_{j,\alpha}e^{\frac{2\pi i}{c\lambda}(l+\kappa_j)d}.\]

\item[(3)] If $-n+\kappa_\alpha<0$, then
\begin{eqnarray*}
a_{n,\alpha,\chi,\rho}(l,j) &=& \frac{2\pi i^{-k}}{\lambda}\sum_{\gamma=\sm a&b\\c&d\esm\in C^+}c^{-1}\biggl(\frac{l+\kappa_j}{n-\kappa_\alpha}\biggr)^{\frac{k-1}2}\chi(\gamma)^{-1}\rho(\gamma^{-1})_{j,\alpha}e^{\frac{2\pi i}{c\lambda}((-n+\kappa_\alpha)a+(l+\kappa_j)d)}\\
&&\times I_{k-1}\biggl(\frac{4\pi}{c\lambda}\sqrt{(n-\kappa_\alpha)(l+\kappa_j)}\biggr).
\end{eqnarray*}
\end{itemize}
\end{thm}

\begin{proof} [\bf Proof of Theorem \ref{fourierpoincare}]
We only sketch the proof since the following results are all either well known or fairly straightforward generalizations of results given in \cite{KM, Leh, Ran}.
By assumption that $\rho(-I_2) = I_p$, we see that $\rho(\pm \gamma) = \rho(\gamma)$. It follows from this that $\gamma$ and $-\gamma$ make equal contributions to the right hand side of (\ref{poincare}). Noting that the terms with $c=0$ correspond to $\gamma = \pm I_2$ we obtain
\begin{eqnarray*}
(P_{n,\alpha,\chi,\rho})_j(\tau) &=& \frac12 \sum_{\gamma} e^{2\pi i(n+\kappa_\alpha)\gamma\tau/\lambda}\chi(\gamma)^{-1}(c\tau+d)^{-k}\rho(\gamma^{-1})_{j,\alpha}\\
&=& \delta_{j,\alpha}e^{2\pi i(n+\kappa_\alpha)\tau/\lambda} + \sum_{\gamma\neq \pm I}e^{2\pi i(n+\kappa_\alpha)\gamma\tau/\lambda}\chi(\gamma)^{-1}(c\tau+d)^{-k}\rho(\gamma^{-1})_{j,\alpha}.
\end{eqnarray*}
From this expression  we use the standard method of invoking the Lipschitz summation formula \cite{KR} (see, e.g., pp. 295-299 of \cite{Leh} or pp. 155-164 of \cite{Ran}):
If $0<\kappa_j<1,\ n>-1$ and $\Im(\tau)>0$, then
\[\sum_{m=0}^\infty (m+\kappa_j)^ne^{2\pi i\tau(m+\kappa_j)} = \frac{\Gamma(n+1)}{(2\pi)^{n+1}}\sum_{q=-\infty}^{\infty}e^{2\pi iq\kappa_j}(-i(\tau-q))^{-n-1}.\]
If $\kappa_j=n=0$ and $\Im(\tau)>0$, then
\[\sum_{m=0}^\infty (m+\kappa_j)^ne^{2\pi i\tau(m+\kappa_j)} = -\frac12 + \frac{1}{2\pi}\sum_{q=-\infty}^{\infty}e^{2\pi iq\kappa_j}(-i(\tau-q))^{-1},\]
where $\sum_{q=-\infty}^{\infty} = \lim_{N\to\infty} \sum_{q=-N}^N$ in the case of $\kappa_j = n =0$.
      Then
we can derive formulas for the Fourier coefficients of $(P_{n,\alpha,\chi,\rho})_j(\tau)$ at $i\infty$  by
using the Bessel functions of the first kind : $I_n(z)$ and $J_n(z)$.
\end{proof}

Now we look at the properties of the Poincar\'e series $P_{n,\alpha,\chi,\rho}(\tau)$.
\begin{thm} \label{othercusp} The Poincar\'e series $P_{n,\alpha,\chi,\rho}(\tau)$ satisfies the following properties:
\begin{enumerate}
\item[(1)] The function $P_{n,\alpha,\chi,\rho}(\tau)$ vanishes at all cusps of $\Gamma$ which are not equivalent to $i\infty$.
\item[(2)] $S_{k,\chi,\rho}(\Gamma)$ is spanned by Poincar\'e series $P_{n,\alpha,\chi,\rho}(\tau)$ with $n+\kappa_\alpha>0$.
\end{enumerate}
\end{thm}

\begin{proof} [\bf Proof of Theorem \ref{othercusp}]
Let $q$ be a parabolic cusp of $\Gamma$ which is not equivalent to $i\infty$ and $\gamma_q = \sm A&B\\C&D\esm\in\SL_2(\RR)$ such that $\gamma_q \infty = q$ as in (2) of Definition \ref{dfnofvvm} . Then we need to check the behavior of the following function at $i\infty$
\[\biggl|(C\tau+D)^{-k}P_{n,\alpha,\chi,\rho}(\gamma_q\tau)\biggr| =   \frac12 \sum_{\gamma=\sm a&b\\c&d\esm} \biggl|\frac{e^{2\pi i(n+\kappa_\alpha)\gamma \gamma_q\tau/\lambda}}{\chi(\gamma)(c'\tau+d')^k}\rho(\gamma^{-1})_{j,\alpha}\biggr| \leq \frac12 \sum_{\gamma=\sm a&b\\c&d\esm} \biggl|\frac{e^{2\pi i(n+\kappa_\alpha)\gamma \gamma_q\tau/\lambda}}{(c'\tau+d')^k}\biggr| \]
for each $1\leq j\leq p$, where $\sm a&b\\c&d\esm \gamma_q = \sm *&*\\ c'&d'\esm$ and $\gamma$ ranges over a set of coset representatives for $<T>\setminus\Gamma$.
Now we estimate the exponential term. This estimate is
a standard one, which was also used in the proof of Proposition \ref{convergence}. Setting $\tau = u+iv$, we see that there is a constant $C_0$ such that
$|e^{2\pi i(n+\kappa_\alpha)\gamma \gamma_q\tau}| \leq e^{2\pi(|n|+1)/(C_0v)}$
for $c'\neq0$.  Since $q$ is not equivalent to $i\infty$, $c'\neq0$ for every $\gamma=\sm a&b\\c&d\esm\in\Gamma$. Therefore, each term $\frac{e^{2\pi i(n+\kappa_\alpha)\gamma \gamma_q\tau/\lambda}}{(c'\tau+d')^k}$ vanishes as $\tau\to i\infty$. This completes the proof of (1).

For the second assertion, we give the definition of the Petersson inner product for vector-valued modular forms: For $f(\tau) = \sum_{j=1}^p f_j(\tau)$ and $g(\tau)= \sum_{j=1}^p g_j(\tau)$ in $M_{k,\chi,\rho}(\Gamma)$, we let
\[(f,g) := \int_{\mc{F}} \sum_{j=1}^p f_j(\tau)\overline{g_j(\tau)}v^k\frac{dudv}{v^2},\]
where $\mc{F}$ is the fundamental domain for $\Gamma$.
If one of $f(\tau)$ and $g(\tau)$ is a cusp form and $\rho$ is unitary, then this integral is well-defined.
Let $\mc{S}\subseteq S_{k,\chi,\rho}(\Gamma)$ be the subspace spanned by the Poincar\'e series $P_{n,\alpha,\chi,\rho}(\tau)$ for which $-n+\kappa_\alpha>0$ and $g(\tau)\in \mc{S}^\perp$. Then $<P_{n,\alpha,\chi,\rho},g> = 0$ for every $n$ and $\alpha$. This implies that the $(-n)$th Fourier coefficients of the $\alpha$th component function $g_\alpha(\tau)$ vanishes for every $n$ and $\alpha$ (see Lemma \ref{petersson}) . So we must have $g(\tau) =0$, and this shows that $\mc{S}^\perp = 0$. Then
 we have $S_{k,\chi,\rho}(\Gamma) = \mc{S}$, which shows that $(2)$ holds.
\end{proof}

\subsection{Supplementary functions} \label{section2.3}
Suppose $f(\tau)\in S_{k+2,\chi,\rho}(\Gamma)$ with $k>0$. By (2) of Theorem \ref{othercusp}, there exist complex numbers $b_1,\cdots, b_s$ such that $f(\tau) = \sum_{i=1}^s b_i P_{n_i,\alpha_i,\chi,\rho}(\tau)$. Put $f^*(\tau) = \sum_{i=1}^s\overline{b_i}P_{n'_i,\alpha_i,\bar{\chi},\bar{\rho}}(\tau)$, where
\begin{equation*}
n'_i =
\begin{cases}
-n_i & \text{if $\kappa_\alpha =0$},\\
1 - n_i & \text{if $\kappa_\alpha>0$}.
\end{cases}
\end{equation*}
Note that
\[\chi(T)\rho(T) = \sm e^{2\pi i\kappa_1}& & & & \\
						                 & \cdot &&&\\
						                 && \cdot &&\\
						                 &&& \cdot &\\
						                 &&&& e^{2\pi i\kappa_p}\esm,\]
where $0\leq \kappa_j<1$ for $1\leq j\leq p$.  Let
\[\bar{\chi}(T)\bar{\rho}(T) = \sm e^{2\pi i\kappa'_1}& & & & \\
						                 & \cdot &&&\\
						                 && \cdot &&\\
						                 &&& \cdot &\\
						                 &&&& e^{2\pi i\kappa'_p}\esm,\]
with $0\leq\kappa_j'<1$ for $1\leq j\leq p$. Then
\begin{equation*}
\kappa_j' =
\begin{cases}
0 & \text{if $\kappa_j=0$},\\
1-\kappa & \text{if $\kappa_j>0$},
\end{cases}
\end{equation*}
for $1\leq j\leq p$.
Thus we have the expansion at $i\infty$
\begin{eqnarray*}
P_{n'_i,\alpha_i,\bar{\chi},\bar{\rho}}(\tau) &=& e^{2\pi i(-n'_i+\kappa'_{\alpha_i})\tau/\lambda}\mbf{e}_{\alpha_i}+\sum_{j=1}^p \sum_{l+\kappa_j>0} a_{n'_i,\alpha_i}(l,j)e^{2\pi i(l+\kappa_j)\tau/\lambda}\mbf{e}_j\\
&=& e^{2\pi i(n_i-\kappa_{\alpha_i})\tau/\lambda}\mbf{e}_{\alpha_i} + \sum_{j=1}^p \sum_{l+\kappa_j>0} a_{n'_i,\alpha_i}(l,j)e^{2\pi i(l+\kappa_j)\tau/\lambda}\mbf{e}_j.
\end{eqnarray*}
It follows that $f^*(\tau)\in M^!_{k+2,\bar{\chi},\bar{\rho}}(\Gamma)$ and $f^*(\tau)$ vanishes at all of the other cusps of $\Gamma$.
Furthermore,  $f^*(\tau)$ has a pole at $i\infty$ with principal part
\[\sum_{i=1}^s \overline{b_i}e^{2\pi i(n_i-\kappa_{\alpha_i})\tau/\lambda}\mbf{e}_{\alpha_i}.\]
We call $f^*(\tau)$ the {\it{function supplementary to}} $f(\tau)$.

Functions $f(\tau)$ and $f^*(\tau)$ have important relations which can be expressed in terms of period functions.  A form $f(\tau)\in M^!_{k+2,\chi,\rho}(\Gamma)$ is a {\it{vector-valued weakly holomorphic cusp form}} if its constant term vanishes. Let $S_{k+2,\chi,\rho}^!(\Gamma)$ denote the space of vector-valued weakly holomorphic cusp forms. Suppose that $f(\tau) = \sum_{j=1}^p\sum_{n\gg-\infty}a(n,j)e^{2\pi i(n+\kappa_j)\tau/\lambda}\mbf{e}_j \in S_{k+2,\chi,\rho}^!(\Gamma)$.
Now we recall two kinds of {\it{Eichler integrals}} of $f(\tau)$:
\begin{eqnarray*}
\mc{E}_f(\tau) &:=& \sum_{j=1}^p\sum_{n\gg-\infty\atop n+\kappa_j\neq0} a(n,j)\biggr(\frac{n+\kappa_j}{\lambda}\biggl)^{-(k+1)}e^{2\pi i(n+\kappa_j)\tau/\lambda}\mbf{e}_j,\\
\mc{E}^N_f(\tau) &:=& \frac{1}{c_{k+2}}\biggl[\int^{i\infty}_\tau f(z)(\bar{\tau}-z)^{k}dz\biggr]^-,
\end{eqnarray*}
where $c_k = -\frac{(k-2)!}{(2\pi i)^{k-1}}$ and $[\ ]^-$ indicates the complex conjugate of the function inside $[\ ]^-$. The Eichler integral $\mc{E}^N_f(\tau)$ is only for a cusp form $f(\tau)$.

We introduce the {\it{period functions}} for $f(\tau)$ by
\begin{eqnarray*}
r(f,\gamma;\tau) &:=& c_{k+2}(\mc{E}_f - \mc{E}_f|_{-k,\chi,\rho}\gamma)(\tau),\\
r^N(f,\gamma;\tau) &:=& c_{k+2}(\mc{E}^N_f - \mc{E}^N_f|_{-k,\bar{\chi},\bar{\rho}}\gamma)(\tau),
\end{eqnarray*}
where $\gamma\in\Gamma$.

In \cite[Theorem 1]{Leh0}, Lehner showed that the Fourier coefficients of modular forms of negative weight are completely determined by the principal part of the expansion of those forms at the cusps using the circle method. Hence, we can define the constant term $c_f$ associated with $\mc{E}_f(\tau)$ using the principal part of $\mc{E}_f(\tau)$. For example, if we assume that $f(\tau)$ has a pole at $i\infty$ and that it is holomorphic at all other cusps, then $c_f$ is equal to
\begin{equation} \label{cf}
c_f := \sum_{j=1}^p \delta_{\kappa_j,0}\biggl(\frac{1}{\lambda(k+1)!}\sum_{t=1}^p\sum_{l<0}\sum_{\gamma=\sm a&b\\c&d\esm\in C^+}a(l,t) \biggl(\frac{-2\pi i}{c}\biggr)^{k+2}\chi^{-1}(\gamma)\rho(\gamma^{-1})_{j,t}e^{\frac{2\pi i}{c\lambda}(l+\kappa_t)a}\biggr)\mbf{e}_j.
\end{equation}
Since we only deal with the case where $f(\tau)$ has a pole only at $i\infty$ in this paper, it is enough to have a formula for $c_f$ as in (\ref{cf}).
With this constant term $c_f$, we define another Eichler integral and period functions of $f(\tau)$ as follows:
\begin{eqnarray*}
\mc{E}^H_f(\tau) &:=& \mc{E}_f(\tau) + c_f,
\end{eqnarray*}
and
\[r^H(f,\gamma;\tau) = c_{k+2}(\mc{E}^H_f - \mc{E}^H_f|_{-k,\chi,\rho}\gamma)(\tau).\]

The period functions of supplementary functions were investigated by Knopp \cite{Kno} for scalar-valued modular forms of integer weights on a $H$-group and by Gim\'enez \cite{Gim} for vector-valued modular forms of integer weights on $\SL_2(\ZZ)$. Following the argument in \cite{Kno} and \cite{Gim}, we obtain a connection between the period function of $f(\tau)$ and
that of its supplementary function $f^*(\tau)$.

\begin{thm} \label{suppleperiod} Suppose that $k$ is a positive integer and $f(\tau)\in S_{k+2,\chi,\rho}(\Gamma)$.
Then 
\begin{equation} \label{equalitysupple}
 r^H(f,\gamma;\tau) = [r^H(f^*,\gamma;\bar{\tau})]^- 
\end{equation}
for all $\gamma\in\Gamma$.
\end{thm}

\begin{proof} [\bf{Proof of Theorem \ref{suppleperiod}}]
Since $\{P_{n,\alpha,\chi,\rho}(\tau)|\ -n+\kappa_\alpha>0 \}$ generates $S_{k+2,\chi,\rho}(\Gamma)$, it is enough to check the equation (\ref{equalitysupple}) for $P_{n,\alpha,\chi,\rho}(\tau)$ with $-n+\kappa_{\alpha}>0$.
Its Fourier expansion is of the form
\[P_{n,\alpha,\chi,\rho}(\tau) = e^{2\pi i(-n+\kappa_{\alpha})\tau/\lambda}\mbf{e}_\alpha + \sum_{j=1}^p\sum_{l+\kappa_j>0}a_{n,\alpha,\chi,\rho}(l,j)e^{2\pi i(l+\kappa_j)\tau/\lambda}\mbf{e}_j,\]
where the Fourier coefficients $a_{n,\alpha,\chi,\rho}(l,j)$ is given in Theorem \ref{fourierpoincare}. Then its Eichler integral $\mc{E}_{P_{n,\alpha,\chi,\rho}}^H(\tau)$ is equal to
\begin{eqnarray*}
 \biggl(\frac{-n+\kappa_\alpha}{\lambda}\biggr)^{-(k+1)}e^{2\pi i(-n+\kappa_\alpha)\tau/\lambda}\mbf{e}_\alpha + \sum_{j=1}^p\sum_{l+\kappa_j>0}a_{n,\alpha,\chi,\rho}(l,j)\biggl(\frac{l+\kappa_j}{\lambda}\biggr)^{-(k+1)}e^{2\pi i(l+\kappa_j)\tau/\lambda}\mbf{e}_j.
\end{eqnarray*}
If we use Theorem \ref{fourierpoincare} and the definition of the Bessel function $J_{k+1}(z)$, then we have
\begin{eqnarray*}
\mc{E}_{P_{n,\alpha,\chi,\rho}}^H(\tau) &=& \biggl(\frac{-n+\kappa_\alpha}{\lambda}\biggr)^{-(k+1)}e^{2\pi i(-n+\kappa_\alpha)\tau/\lambda}\mbf{e}_\alpha+
 \sum_{j=1}^p\frac{2\pi i^{-(k+2)}}{\lambda^{-k}}\sum_{\gamma=\sm a&b\\c&d\esm\in C^+} c^{-1}\chi(\gamma)^{-1}\rho(\gamma^{-1})_{j,\alpha}\\
&&\times e^{\frac{2\pi i}{c\lambda}(-n+\kappa_\alpha)a}\sum_{t=0}^\infty \frac{(-1)^t(-n+\kappa_\alpha)^t}{t!\Gamma(t+k+2)}\biggl(\frac{2\pi}{c\lambda}\biggr)^{2t+k+1} \sum_{l+\kappa_j>0}(l+\kappa_j)^te^{\frac{2\pi i}{c\lambda}(l+\kappa_j)d}\mbf{e}_j.
\end{eqnarray*}
If we use the Lipschitz summation formula as in the proof of Theorem \ref{fourierpoincare}, we see that
\begin{eqnarray*}
&&\sum_{t=0}^\infty \frac{(-1)^t(-n+\kappa_\alpha)^t}{t!\Gamma(t+k+2)}\biggl(\frac{2\pi}{c\lambda}\biggr)^{2t+k+1}\sum_{l+\kappa_j>0}(l+\kappa_j)^te^{\frac{2\pi i}{c\lambda}(l+\kappa_j)d}\\
&& = \delta_{\kappa_j,0}\frac{-1}{2\Gamma(k+2)}\biggl(\frac{2\pi}{c\lambda}\biggr)^{k+1}
+ \sum_{t=0}^\infty \frac{(-1)^t(-n+\kappa_\alpha)^t}{t!\Gamma(t+k+2)}\biggl(\frac{2\pi}{c\lambda}\biggr)^{2t+k+1}\frac{\Gamma(t+1)}{(2\pi)^{t+1}}\\
&&\times\lim_{N\to\infty} \sum_{q=-N}^N e^{2\pi iq\kappa_j}\biggl(-i\biggl(\frac{\tau}{\lambda}+\frac {d}{c\lambda}-q\biggr)\biggr)^{-t-1}.
\end{eqnarray*}
 Therefore, we have that the Eichler integral $\mc{E}_{P_{n,\alpha,\chi,\rho}}^H(\tau)$ equals to
\begin{eqnarray*}
&&  \biggl(\frac{-n+\kappa_\alpha}{\lambda}\biggr)^{-(k+1)}e^{2\pi i(-n+\kappa_\alpha)\tau/\lambda}\mbf{e}_\alpha+\frac{(-1)^{k+1}}{2}\overline{c_{P_{n',\alpha,\bar{\chi},\bar{\rho}}}}\\
&&+\biggl(\frac{-n+\kappa_\alpha}{\lambda}\biggr)^{-(k+1)} \sum_{j=1}^p\sum_{\gamma=\sm a&b\\c&d\esm\in C^+}\chi(\gamma)^{-1}\rho(\gamma^{-1})_{j,\alpha}e^{\frac{2\pi i}{c\lambda}(-n+\kappa_\alpha)a}\\
&&\times
\lim_{N\to\infty} \sum_{q=-N}^N e^{2\pi iq\kappa_j}(c\tau+d-c\lambda q)^k \sum_{t=k+1}^\infty \frac{1}{k!}\biggl(\frac{2\pi i(n-\kappa_\alpha)}{c\lambda(c\tau+d-c\lambda q)}\biggr)^t\mbf{e}_j,
\end{eqnarray*}
where
\[c_{P_{n',\alpha,\bar{\chi},\bar{\rho}}} = \sum_{j=1}^p \delta_{\kappa'_j,0}\biggl(\frac{1}{\lambda(k+1)!}\sum_{\gamma=\sm a&b\\c&d\esm\in C^+}\biggl(\frac{-2\pi i}{c}\biggr)^{k+2}\bar{\chi}^{-1}(\gamma)\bar{\rho}(\gamma^{-1})_{j,\alpha}e^{\frac{2\pi i}{c\lambda}(-n'+\kappa'_\alpha)a}\biggr)\mbf{e}_j.\]
Let $D = \{\sm a&b\\c&d\esm\in\Gamma|\ c>0,\ 0\leq a<c\lambda\}$. Then we have
\begin{eqnarray*}
\mc{E}_{P_{n,\alpha,\chi,\rho}}^H(\tau) &=& \biggl(\frac{-n+\kappa_\alpha}{\lambda}\biggr)^{-(k+1)}e^{2\pi i(-n+\kappa_\alpha)\tau/\lambda}\mbf{e}_\alpha +\frac{(-1)^{k+1}}{2}\overline{c_{P_{n',\alpha,\bar{\chi},\bar{\rho}}}}\\
&& + \biggl(\frac{-n+\kappa_\alpha}{\lambda}\biggr)^{-(k+1)} \sum_{j=1}^p\lim_{N\to\infty} \sum_{\gamma = \sm a&b\\c&d\esm\in D\atop |d|\leq N} \chi(\gamma)^{-1}\rho(\gamma^{-1})_{j,\alpha}e^{\frac{2\pi i}{c\lambda}(-n+\kappa_\alpha)a}
(c\tau+d)^k\\
&&\times \sum_{t=k+1}^\infty \frac{1}{k!}\biggl(\frac{2\pi i(n-\kappa_\alpha)}{c\lambda(c\tau+d)}\biggr)^t\mbf{e}_j.
\end{eqnarray*}
Since $k$ is an integer, we see that
\begin{equation} \label{exponential}
 \sum_{t=k+1}^\infty \frac{1}{k!}\biggl(\frac{2\pi i(n-\kappa_\alpha)}{c\lambda(c\tau+d)}\biggr)^t = e^{\frac{2\pi i(n-\kappa_\alpha)}{c\lambda(c\tau+d)}} - \sum_{t=0}^k \frac{1}{k!}\biggl(\frac{2\pi i(n-\kappa_\alpha)}{c\lambda(c\tau+d)}\biggr)^t
\end{equation}
and
\begin{equation} \label{exponential2}
e^{\frac{2\pi i}{c\lambda}(-n+\kappa_\alpha)a} e^{\frac{2\pi i(n-\kappa_\alpha)}{c\lambda(c\tau+d)}} = e^{-2\pi i(n-\kappa_\alpha)\gamma\tau/\lambda}.
\end{equation}
After some computations using (\ref{exponential}) and (\ref{exponential2}), we obtain that the Eichler integral $\mc{E}_{P_{n,\alpha,\chi,\rho}}^H(\tau)$ is the same as
\begin{eqnarray}\label{eichler1}
&& \biggl(\frac{-n+\kappa_\alpha}{\lambda}\biggr)^{-(k+1)}e^{2\pi i(-n+\kappa_\alpha)\tau/\lambda}\mbf{e}_\alpha +\frac{(-1)^{k+1}}{2}\overline{c_{P_{n',\alpha,\bar{\chi},\bar{\rho}}}}\\
\nonumber &&+ \biggl(\frac{-n+\kappa_\alpha}{\lambda}\biggr)^{-(k+1)}\sum_{j=1}^p \lim_{K\to\infty}\biggl\{ \sum_{\gamma = \sm a&b\\c&d\esm\in D\atop c\leq TK,\ |d|\leq K}\chi(\gamma)^{-1}\rho(\gamma^{-1})_{j,\alpha}(c\tau+d)^k e^{-2\pi i(n-\kappa_\alpha)\gamma\tau/\lambda}\\
\nonumber &&- \sum_{\gamma = \sm a&b\\c&d\esm\in D\atop c\leq TK,\ |d|\leq K}\chi(\gamma)^{-1}\rho(\gamma^{-1})_{j,\alpha}(c\tau+d)^k e^{\frac{2\pi i}{c\lambda}(-n+\kappa_\alpha)a}\sum_{t=0}^k \frac{1}{k!}\biggl(\frac{2\pi i(n-\kappa_\alpha)}{c\lambda(c\tau+d)}\biggr)^t\biggr\}\mbf{e}_j,
\end{eqnarray}
where $T$ is a positive integer.
By the similar way, we see that  Eichler integral $\mc{E}_{P_{n',\alpha,\bar{\chi},\bar{\rho}}}^H(\tau)$ is equal to
\begin{eqnarray} \label{eichler2}
&& \biggl(\frac{-n'+\kappa'_\alpha}{\lambda}\biggr)^{-(k+1)}e^{2\pi i(-n'+\kappa'_\alpha)\tau/\lambda}\mbf{e}_\alpha +\frac{1}{2}c_{P_{n',\alpha,\bar{\chi},\bar{\rho}}}\\
\nonumber &&+ \biggl(\frac{-n'+\kappa'_\alpha}{\lambda}\biggr)^{-(k+1)}\sum_{j=1}^p \lim_{K\to\infty}\biggl\{ \sum_{\gamma = \sm a&b\\c&d\esm\in D\atop c\leq TK,\ |d|\leq K}\bar{\chi}(\gamma)^{-1}\bar{\rho}(\gamma^{-1})_{j,\alpha}(c\tau+d)^k e^{-2\pi i(n'-\kappa'_\alpha)\gamma\tau/\lambda}\\
\nonumber &&- \sum_{\gamma = \sm a&b\\c&d\esm\in D\atop c\leq TK,\ |d|\leq K}\bar{\chi}(\gamma)^{-1}\bar{\rho}(\gamma^{-1})_{j,\alpha}(c\tau+d)^k e^{\frac{2\pi i}{c\lambda}(-n'+\kappa'_\alpha)a}\sum_{t=0}^k \frac{1}{k!}\biggl(\frac{2\pi i(n'-\kappa'_\alpha)}{c\lambda(c\tau+d)}\biggr)^t\biggr\}\mbf{e}_j.
\end{eqnarray}

Note that $P_{n',\alpha,\bar{\chi},\bar{\rho}}(\tau)$ is a function supplementary to $P_{n,\alpha,\chi,\rho}(\tau)$.
Using the expressions of the Eichler integrals $\mc{E}_{P_{n,\alpha,\chi,\rho}}^H(\tau)$ and $\mc{E}_{P_{n',\alpha,\bar{\chi},\bar{\rho}}}^H(\tau)$ as in (\ref{eichler1}) and (\ref{eichler2}) and the fact that $-n' + \kappa'_j = n - \kappa_j$, then we see that
\[r^H(P_{n',\alpha,\bar{\chi},\bar{\rho}},\gamma;\bar{\tau})]^- = r^H(P_{n,\alpha,\chi,\rho},\gamma;\tau)\]
for all $\gamma\in \Gamma$, which completes the proof.
\end{proof}

On the other hand, it is well known that if $f(\tau)\in S_{k+2,\chi,\rho}(\Gamma)$, then
\[\mc{E}_f^H(\tau) = \frac{1}{c_{k+2}}\int^{i\infty}_{\tau}f(z)(\tau-z)^kdz.\]
Furthermore, we can compute $r^H(f,\gamma;\tau)$ and $r^N(f,\gamma;\tau)$ explicitly using the integral representation of those (see \cite{KM2}):
\begin{eqnarray*}
r^H(f,\gamma;\tau) &=& \int^{i\infty}_{\gamma^{-1}(i\infty)}f(z)(\tau-z)^kdz,\\
r^N(f,\gamma;\tau) &=& \biggl[\int^{i\infty}_{\gamma^{-1}(i\infty)}f(z)(\bar{\tau}-z)^kdz\biggr]^-.
\end{eqnarray*}
From this formula, we obtain the following result.

\begin{thm} \cite[Section 2]{HK} \label{HNperiod} Suppose that $k$ is a positive integer and $f(\tau)\in S_{k+2,\chi,\rho}(\Gamma)$.
Then 
\[r^H(f,\gamma;\tau) = [r^N(f,\gamma;\bar{\tau})]^-.\]
\end{thm}

\section{$L$-functions}\label{section3} In this section, we recall the theory of $L$-functions associated to scalar-valued weakly holomorphic modular forms using the definition in \cite{BFK} and extend this theory to general $H$-groups. Let $\Gamma$ be a $H$-group. Now consider $f(\tau) = \sum_{m\gg-\infty} a(m)e^{2\pi i(m+\kappa)\tau/\lambda}\in S^!_{k,\chi}(\Gamma)$, where $S^!_{k,\chi}(\Gamma)$ denotes the space of weakly holomorphic cusp forms of weight $k$ and character $\chi$ on $\Gamma$.

Let $\gamma = \sm a&b\\c&d\esm\in\Gamma$ and let $t_0>0$. Then  a twisted $L$-series for $f(\tau)$ is defined by
\begin{eqnarray} \label{dfntwisted}
L(f,\zeta_{c\lambda}^{-d},s) &:=& \frac{(2\pi)^s}{\Gamma(s)}L^*(f,\zeta_{c\lambda}^{-d},s),\\
\nonumber L^*(f,\zeta_{c\lambda}^{-d},s) &:=& \sum_{m\gg-\infty} \frac{a(m)\zeta^{-d(m+\kappa)}_{c\lambda}\Gamma\biggl(s,\frac{2\pi (m+\kappa)t_0}\lambda\biggr)}{\biggl(\frac{2\pi (m+\kappa)}\lambda\biggr)^s}\\
\nonumber && + \chi^{-1}(\gamma)i^k(-c)^{k-2s}\sum_{m\gg-\infty} \frac{a(m)\zeta_{c\lambda}^{a(m+\kappa)}\Gamma\biggl(k-s, \frac{2\pi (m+\kappa)}{c^2t_0\lambda}\biggr)}{\biggl(\frac{2\pi (m+\kappa)}\lambda\biggr)^{k-s}}.
\end{eqnarray}
It is known that this definition is independent of $t_0$ (see the proof of Theorem \ref{integralrep}).
Here, the incomplete gamma function $\Gamma(s,z)$ is given by the analytic continuation (to an entire function with respect to $s$ and fixed $z\neq 0$) of $\int_{z}^\infty e^{-t}t^{s-1}dt$ and $\tau^s = |\tau|e^{i\arg(\tau)s}$, $-\pi < \arg(\tau) \leq \pi$.

To give an integral representation of $L(f,\zeta_{c\lambda}^{-d},s)$,
 we require certain regularized integrals. For this, consider a continuous function $f:\HH\to\CC$ such that $f(\tau) = O(e^{cv})$ for some constant $c>0$ uniformly in $u$ as $v\to\infty$. Then, for each $\tau_0\in\HH$, the integral
\[\int^{i\infty}_{\tau_0} e^{tiw}f(w)dw\]
is convergent for $t\in\CC$ with $\Re(t)\gg0$, where the path of integration lies within a vertical strip. If this integral has an analytic continuation to $t=0$,  the {\it{regularized integral}} is defined by
\[R.\int^{i\infty}_{\tau_0}f(w)dw := \biggl[\int^{i\infty}_{\tau_0} e^{tiw}f(w)dw\biggr]_{t=0},\]
where the right hand side means that we take the value at $t=0$ of the analytic continuation of the integral. Similarly,  integrals at other cusps $q$ also can be defined. Specifically, suppose that $q = \gamma_q(i\infty)$ for $\gamma_q = \sm a&b\\c&d\esm\in \SL_2(\RR)$. If $f(\gamma_q\tau) = O(e^{cv})$, then we let
\[R. \int^{q}_{\tau_0}f(w)dw := R. \int^{i\infty}_{\gamma_q^{-1}\tau_0}(cw+d)^{-2}f(\gamma_q w)dw.\]
For cusps $q_1,q_2$, we set
\begin{equation} \label{cuspregular}
R. \int_{q_1}^{q_2} f(w)dw := R. \int_{\tau_0}^{q_2}f(w)dw + R. \int_{q_1}^{\tau_0} f(w)dw
\end{equation}
for any $\tau_0\in\HH$. One can see that this integral is independent of $\tau_0\in\HH$. The following result gives an integral representation of $L(f,\zeta_{c\lambda}^{-d},s)$. 

\begin{thm} \label{integralrep} Assume that $f(\tau)\in S^!_{k,\chi}(\Gamma)$. We have the identity
\begin{equation} \label{integralrep2}
L^*(f,\zeta_{c\lambda}^{-d},s) = i^{-s}R.\int^{i\infty}_{-\frac dc}f(\tau)\biggl(\tau+\frac dc\biggr)^{s-1}d\tau
\end{equation}
for $\gamma = \sm a&b\\c&d\esm\in\Gamma$.
\end{thm}

\begin{proof} [\bf Proof of Theorem \ref{integralrep}] Fix $t_0>0$. Then 
by the definition of the regularized integral for two cusps as in (\ref{cuspregular})
we can divide the integral in (\ref{integralrep2}) into two parts:
\begin{eqnarray} \label{twopart}
i^{-s}R.\int^{i\infty}_{it_0-\frac dc}f(\tau)\biggl(\tau+\frac dc\biggr)^{s-1}d\tau + i^{-s}R.\int^{it_0-\frac dc}_{-\frac dc}f(\tau)\biggl(\tau+\frac dc\biggr)^{s-1}d\tau.
\end{eqnarray}
Inserting the Fourier expansion of $f(\tau) = \sum_{m\gg-\infty} a(m)e^{2\pi i(m+\kappa)\tau/\lambda}$ yields that we have
\begin{eqnarray} \label{integral2}
\nonumber i^{-s}R.\int^{i\infty}_{it_0-\frac dc}f(\tau)\biggl(\tau+\frac dc\biggr)^{s-1}d\tau&=&i^{-s}R.\int^{i\infty}_{it_0} f\left(\tau-\frac dc\right)\tau^{s-1}d\tau\\
\nonumber &=& i^{-s}\left[\int^{i\infty}_{it_0} e^{ui\tau} \sum_{m\gg-\infty} a(m) e^{2\pi i(m+\kappa)\tau/\lambda}\zeta_{c\lambda}^{-(m+\kappa)d}\tau^{s-1}d\tau\right]_{u=0}\\
 &=& i^{-s}\sum_{m\gg-\infty} a(m) \left[\int^{i\infty}_{it_0} e^{ui\tau} e^{2\pi i(m+\kappa)\tau/\lambda}\zeta_{c\lambda}^{-(m+\kappa)d}\tau^{s-1}d\tau\right]_{u=0}.
\end{eqnarray}
Using the definition of regularized integral and change of variable $\tau\mapsto it$, we see that (\ref{integral2}) is  equal to
\begin{eqnarray*}
&& \sum_{m\gg-\infty} a(m)\biggl[\int^{\infty}_{t_0}e^{-2\pi(m+\kappa)t/\lambda}\zeta_{c\lambda}^{-(m+\kappa)d}t^{s-1}e^{-ut}dt\biggr]_{u=0}\\
&&= \sum_{m\gg-\infty}a(m)\zeta_{c\lambda}^{-(m+\kappa)d}\biggl[\biggl(\frac{2\pi(m+\kappa)}{\lambda}+u\biggr)^{-s}\int^\infty_{\biggl(\frac{2\pi(m+\kappa)}{\lambda}+u\biggr)t_0} e^{-t}t^{s-1}dt\biggr]_{u=0}\\
&&=\sum_{m\gg-\infty} \frac{a(m)\zeta^{-d(m+\kappa)}_{c\lambda}\Gamma\biggl(s,\frac{2\pi (m+\kappa)t_0}\lambda\biggr)}{\biggl(\frac{2\pi (m+\kappa)}\lambda\biggr)^s}.
\end{eqnarray*}

For the second integral in (\ref{integralrep2}), note that $-\frac dc = \gamma^{-1}(i\infty)$ and that $f(\tau) = (c\tau+d)^{-k}\chi^{-1}(\gamma)f(\gamma\tau)$. One can see that
\begin{eqnarray*}
&& i^{-s}R.\int^{it_0-\frac dc}_{-\frac dc}f(\tau)\biggl(\tau+\frac dc\biggr)^{s-1}d\tau= i^{-s}c^{-k}\chi^{-1}(\gamma)R.\int^{it_0-\frac dc}_{\gamma^{-1}(i\infty)} f(\gamma\tau)\biggl(\tau+\frac dc\biggr)^{s-k-1}d\tau.
\end{eqnarray*}
By the change of variable, this integral equals to
\begin{eqnarray} \label{integral1}
&&i^{-s}c^{1-s}\chi^{-1}(\gamma)R.\int^{\frac ca+\frac{i}{c^2t_0}}_{i\infty}f(\tau)(-c\tau+a)^{k-s-1}d\tau\\
\nonumber &&=i^{-s}c^{1-s}\chi^{-1}(\gamma)R.\int^{\frac{i}{c^2t_0}}_{i\infty}f\biggl(\tau+\frac ac\biggr)(-c\tau)^{k-s-1}d\tau.
\end{eqnarray}
If we use the definition of regularized integral and the change of variable, then we obtain that the integral (\ref{integral1}) is the same as
\begin{eqnarray*}
i^{k-2s}(-1)^{k-s-1}c^{k-2s}\chi^{-1}(\gamma)\biggl[\int^{\frac{1}{c^2t_0}}_{\infty}f\biggl(it+\frac ac\biggr)t^{k-s-1}e^{-ut}dt\biggr]_{u=0}.
\end{eqnarray*}
If we insert the Fourier expansion of $f(\tau)$, then we get the following result by the similar computation as we used to compute the first integral in (\ref{twopart})
\begin{eqnarray*}
&&i^{-s}R.\int^{it_0-\frac dc}_{-\frac dc}f(\tau)\biggl(\tau+\frac dc\biggr)^{s-1}d\tau = \chi^{-1}(\gamma)i^k(-c)^{k-2s}\sum_{m\gg-\infty} \frac{a(m)\zeta_{c\lambda}^{a(m+\kappa)}\Gamma\biggl(k-s, \frac{2\pi (m+\kappa)}{c^2t_0\lambda}\biggr)}{\biggl(\frac{2\pi (m+\kappa)}\lambda\biggr)^{k-s}}.
\end{eqnarray*}
From this we can also see that the definition of $L(f,\zeta_{c\lambda}^{-d},s)$ given in (\ref{dfntwisted}) is independent of the choice of $t_0$.
\end{proof}

Using this integral representation, we can derive a formula for period functions of $f(\tau)\in S_{k,\chi}(\Gamma)$ in terms of critical values of $L$-functions.

\begin{thm}  \label{Lformula} Suppose that $f(\tau)\in S^!_{k+2, \chi}(\Gamma)$ and $\gamma = \sm a&b\\c&d\esm \in\Gamma$. Then
\begin{enumerate}
\item[(1)] $\mc{E}_f(\tau) = \frac{(-1)^k}{c_{k+2}}R.\int^{i\infty}_\tau f(z)(z-\tau)^kdz$,
\item[(2)] $r(f,\gamma;\tau) = \sum_{n=0}^{k} i^{1-n}\mat k\\n\emat  \frac{\Gamma(n+1)}{(2\pi)^{n+1}}L(f,\zeta_{c\lambda}^{-d},n+1)\biggl(\tau+\frac dc\biggr)^{k-n}$.
\end{enumerate}
\end{thm}

\begin{proof} [\bf Proof of Theorem \ref{Lformula}]
By the change of variable $z\mapsto z+\tau$, we have
\begin{eqnarray} \label{mc1}
\frac{(-1)^k}{c_{k+2}}R.\int^{i\infty}_\tau f(z)(z-\tau)^kdz= \frac{(-1)^k}{c_{k+2}}R.\int^{i\infty}_0f(z+\tau)z^kdz.
\end{eqnarray}
Inserting the Fourier expansion of $f(\tau) = \sum_{m\gg-\infty} a(m)e^{2\pi i(m+\kappa)\tau/\lambda}$, we obtain that the integral in (\ref{mc1}) is equal to
\begin{eqnarray*}
&&\frac{(-1)^k}{c_{k+2}}\sum_{m\gg-\infty} a(m)\biggl[e^{2\pi i(m+\kappa)\tau/\lambda}\int^{i\infty}_0
e^{2\pi i(m+\kappa)z/\lambda}e^{uiz}z^kdz\biggr]_{u=0}\\
&&=\frac{(-1)^k i^{k+1}}{c_{k+2}}\sum_{m\gg-\infty}a(m)e^{2\pi i(m+\kappa)\tau/\lambda}\biggl[\int^\infty_0 e^{-2\pi(m+\kappa)t/\lambda}e^{-ut}t^kdt\biggr]_{u=0}\\
&&=\sum_{m\gg-\infty} a(m)\biggl(\frac{m+\kappa}{\lambda}\biggr)^{-(k+1)}e^{2\pi i(m+\kappa)\tau/\lambda}.
\end{eqnarray*}
This shows that the first claim is true.

Note that $r(f,\gamma;\tau) = c_{k+2}(\mc{E}_f- \mc{E}_f|_{-k,\chi}\gamma)(\tau)$.
Then we need to compute the second term $(\mc{E}_f|_{-k,\chi}\gamma)(\tau)$. By the change of variable, we have
\begin{eqnarray*}
(\mc{E}_f|_{-k,\chi}\gamma)(\tau) &=& \frac{(-1)^k}{c_{k+2}}(c\tau+d)^k\chi^{-1}(\gamma)R.\int^{\gamma^{-1}(i\infty)}_{\tau}f(\gamma z)(\gamma z-\gamma\tau)^k(cz+d)^{-2}dz.
\end{eqnarray*}
Since $f(\tau)\in S^!_{k+2,\chi}(\Gamma)$, we obtain that
\begin{eqnarray*}
(\mc{E}_f|_{-k,\chi}\gamma)(\tau) &=& \frac{(-1)^k}{c_{k+2}}R. \int^{\gamma^{-1}(i\infty)}_{\tau} f(z)(c\tau+d)^k(cz+d)^k(\gamma z-\gamma\tau)^kdz\\
&=& \frac{(-1)^k}{c_{k+2}}R. \int^{\gamma^{-1}(i\infty)}_{\tau} f(z)(z-\tau)^kdz.
\end{eqnarray*}
Therefore, by the definition of $\mc{E}_f(\tau)$, we have
\[r(f,\gamma;\tau) = (-1)^k R. \int^{i\infty}_{\gamma^{-1}(i\infty)}f(z)(z-\tau)^kdz.\]
Using the binomial expansion, we can see that $r(f,\gamma;\tau)$ equals to
\begin{eqnarray*}
(-1)^k\sum_{n=0}^k\mat k\\ n\emat (-1)^{k-n}\biggl(\tau+\frac dc\biggr)^{k-n}R.\int^{i\infty}_{-\frac dc}f(z)\biggl(z+\frac dc\biggr)^ndz.
\end{eqnarray*}
Now the second claim follows using the integral representation of $L(f,\zeta_{c\lambda}^{-d},s)$.
\end{proof}

\section{Proofs of the main theorems} \label{section4}
In this section we prove the main results: Theorem \ref{main1}, Corollary \ref{main2} and Theorem \ref{main3}. To prove Theorem \ref{main1}, we construct a vector-valued grid. Then Corollary \ref{main2} follows from Theorem \ref{main1} and the property of Petersson's inner product. Theorem \ref{main3} follows from properties of $L$-functions and supplementary functions.

\begin{proof} [\bf Proof of Theorem \ref{main1}]
Let $P_{n,\alpha,\chi,\rho}(\tau)\in S_{k+2,\chi,\rho}(\Gamma)$ be a Poincar\'e series defined in (\ref{poincare}) for $-n+\kappa_\alpha>0$. Then $P_{n',\alpha,\bar{\chi},\bar{\rho}}(\tau) \in S^!_{k+2,\bar{\chi},\bar{\rho}}(\Gamma)$ is a supplementary function as in section \ref{section2.3}.
Now we construct a vector-valued harmonic weak Maass form using Eichler integrals.
Let
\[G_{n',\alpha,\bar{\chi},\bar{\rho}}(\tau) := G^+_{n',\alpha,\bar{\chi},\bar{\rho}} + G^-_{n',\alpha,\bar{\chi},\bar{\rho}}(\tau),\]
where
\begin{eqnarray*}
G^+_{n',\alpha,\bar{\chi},\bar{\rho}}(\tau) &:=& \biggl(\frac{-n'+\kappa'_\alpha}{\lambda}\biggr)^{k+1}\mc{E}_{P_{n',\alpha,\bar{\chi},\bar{\rho}}}^H(\tau),\\
G^-_{n',\alpha,\bar{\chi},\bar{\rho}}(\tau) &:=& -\biggl(\frac{-n'+\kappa'_\alpha}{\lambda}\biggr)^{k+1}\mc{E}_{P_{n,\alpha,\chi,\rho}}^N(\tau).
\end{eqnarray*}
We need to show that $G_{n',\alpha,\bar{\chi},\bar{\rho}}(\tau)$ is a harmonic weak Maass form in $H^*_{-k, \bar{\chi},\bar{\rho}}(\Gamma)$ by checking the conditions as in Definition \ref{dfnofhar}.

To show that $G_{n',\alpha,\bar{\chi},\bar{\rho}}(\tau)$ is invariant under the slash operator $|_{-k,\bar{\chi},\bar{\rho}}\gamma$ for $\gamma\in\Gamma$, we use the property of Eichler integrals. Fix $\gamma\in \Gamma$. Then, by the definition of $G_{n',\alpha,\bar{\chi},\bar{\rho}}(\tau)$, we have
\begin{eqnarray*}
&&G_{n',\alpha,\bar{\chi},\bar{\rho}}(\tau) - (G_{n',\alpha,\bar{\chi},\bar{\rho}}|_{-k,\bar{\chi},\bar{\rho}}\gamma)(\tau)\\
\nonumber &&= \biggl(\frac{-n'+\kappa'_\alpha}{\lambda}\biggr)^{k+1}\biggl[\mc{E}^H_{P_{n',\alpha,\bar{\chi},\bar{\rho}}}(\tau) - (\mc{E}^H_{P_{n',\alpha,\bar{\chi},\bar{\rho}}}|_{-k.\bar{\chi},\bar{\rho}}\gamma)(\tau) - \mc{E}^N_{P_{n,\alpha,\chi,\rho}}(\tau)+ (\mc{E}^N_{P_{n,\alpha,\chi,\rho}}|_{-k,\bar{\chi},\bar{\rho}}\gamma)(\tau)\biggr]\\
&&= \biggl(\frac{-n'+\kappa'_\alpha}{\lambda}\biggr)^{k+1}\frac{1}{c_{k+2}}\biggl[r^H(P_{n',\alpha,\bar{\chi},\bar{\rho}},\gamma;\tau) - r^N(P_{n,\alpha,\chi,\rho},\gamma;\tau)\biggr].
\end{eqnarray*}
Since $P_{n',\alpha,\bar{\chi},\bar{\rho}}(\tau)$ is a function supplementary to $P_{n,\alpha,\chi,\rho}(\tau)$, by Theorem \ref{suppleperiod} and Theorem \ref{HNperiod}, for any $\gamma\in\Gamma$, we have
\begin{eqnarray*}
r^H(P_{n',\alpha,\bar{\chi},\bar{\rho}},\gamma;\tau) &=&  [r^H(P_{n,\alpha,\chi,\rho},\gamma;\bar{\tau})]^-= r^N(P_{n,\alpha,\chi,\rho},\gamma;\tau)
\end{eqnarray*}
because $P_{n',\alpha,\bar{\chi},\bar{\rho}}(\tau)$ is a function supplementary to $P_{n,\alpha,\chi,\rho}(\tau)$.
This implies that for any $\gamma\in\Gamma$
\[(G_{n',\alpha,\bar{\chi},\bar{\rho}}|_{-k,\bar{\chi},\bar{\rho}}\gamma)(\tau) = G_{n',\alpha,\bar{\chi},\bar{\rho}}(\tau)\]
for all $\gamma\in\Gamma$.

For the second condition, we need to compute $\Delta_{-k}(G_{n',\alpha,\bar{\chi},\bar{\rho}})$. Since $G^+_{n',\alpha,\bar{\chi},\bar{\rho}}(\tau)$ is holomorphic, $\Delta_{-k}(G^+_{n',\alpha,\bar{\chi},\bar{\rho}})$ is equal to zero.
After some computations, we find that
\[\frac{\partial}{\partial\bar{\tau}}\biggl(\mc{E}^N_{P_{n,\alpha,\chi,\rho}}\biggr)(\tau) =  \frac{1}{c_{k+2}}\overline{P_{n,\alpha,\chi,\rho}(\tau)}(2iv)^k\]
and
\[\frac{\partial^2}{\partial\tau\partial\bar{\tau}}\biggl(\mc{E}^N_{P_{n,\alpha,\chi,\rho}}\biggr)(\tau) = \frac{1}{c_{k+2}}\overline{P_{n,\alpha,\chi,\rho}(\tau)}(2iv)^{k-1}.\]
Since the weight $-k$ hyperbolic Laplacian is given by
\[\Delta_{-k} = -4v^2\frac{\partial^2}{\partial\tau\partial\bar{\tau}}-2ikv\frac{\partial}{\partial\bar{\tau}},\]
we can  see that $\Delta_{-k}(G^-_{n',\alpha,\bar{\chi},\bar{\rho}})$ is also equal to zero. Finally, the third condition, which is the growth condition, follows from the fact that $P_{n,\alpha,\chi,\rho}(\tau)$ is a cusp form and $P_{n', \alpha,\bar{\chi},\bar{\rho}}(\tau)$ is a weakly holomorphic modular form.

For the differential operator $\xi_{-k}$, we compute $(\xi_{-k}\mc{E}^N_f)(\tau)$ for a cusp form $f(\tau)\in S_{k+2,\chi,\rho}(\Gamma)$:
\begin{eqnarray*}
(\xi_{-k}\mc{E}^N_f)(\tau) &=& 2iv^{-k}\overline{\frac{\partial}{\partial\bar{\tau}}\biggl(\frac{1}{c_{k+2}}\biggl[\int^{i\infty}_\tau f(z)(\bar{\tau}-z)^kdz\biggr]^-\biggr)}\\
&=& 2iv^{-k}\frac{1}{\overline{c_{k+2}}}\frac{\partial}{\partial\tau}\biggl(\int^{i\infty}_\tau f(z)(\bar{\tau}-z)^kdz\biggr)\\
&=& 2iv^{-k}\frac{1}{\overline{c_{k+2}}}(-f(\tau)(-2iv)^k)= -\frac{(-4\pi)^{k+1}}{k!}f(\tau).
\end{eqnarray*}
Therefore, by the definition of $G_{n',\alpha,\bar{\chi},\bar{\rho}}(\tau)$, we have
\[(\xi_{-k}G_{n',\alpha,\bar{\chi},\bar{\rho}})(\tau) = \frac{(-4\pi)^{k+1}}{k!}\biggl(\frac{-n'+\kappa'_\alpha}{\lambda}\biggr)^{k+1}P_{n,\alpha,\chi,\rho}(\tau).\]
Hence, $G_{n',\alpha,\bar{\chi},\bar{\rho}}(\tau)$ is a harmonic weak Maass form in $H^*_{-k, \bar{\chi},\bar{\rho}}(\Gamma)$.

Our claim is that a pair of families $\{G_{n_2,\alpha_2, \bar{\chi},\bar{\rho}}\}$ and $\{P_{n_1,\alpha_1,\chi,\rho}\}$ form a vector-valued grid of weight $k+2$, character $\chi$ and type $\rho$ on $\Gamma$.
By the definition of $G^+_{n_2,\alpha_2, \bar{\chi},\bar{\rho}}(\tau)$ and $P_{n_1,\alpha_1,\chi,\rho}(\tau)$, one can check that they have the desired principal part in their Fourier expansions at $i\infty$.
It remains to prove the coefficient relation which a grid should satisfy. By the way of construction, we have
\begin{eqnarray*}
G^+_{n_2,\alpha_2,\bar{\chi},\bar{\rho}}(\tau)&=& e^{2\pi i(-n_2+\kappa'_{\alpha_2})\tau/\lambda}\mbf{e}_{\alpha_2} + \sum_{j=1}^p\sum_{l+\kappa'_j>0}a_{n_2,\alpha_2,\bar{\chi},\bar{\rho}}(l,j)\biggl(\frac{-n_2+\kappa'_{\alpha_2}}{l+\kappa'_j}\biggr)^{k+1}e^{2\pi i(l+\kappa'_j)\tau/\lambda}\mbf{e}_j\\
&+&\sum_{j=1}^p \frac{\delta_{\kappa_j,0}i^k}{(k+1)!}\sum_{\gamma=\sm a&b\\c&d\esm\in C^+}\biggl(\frac{2\pi}{c\lambda}\biggr)^{k+2}\bar{\chi}^{-1}(\gamma)\bar{\rho}(\gamma^{-1})_{j,\alpha_2}e^{\frac{2\pi i}{c\lambda}(-n_2+\kappa'_{\alpha_2})a}(n_2-\kappa'_{\alpha_2})^{k+1}\mbf{e}_j,
\end{eqnarray*}
where $a_{n_2,\alpha_2,\bar{\chi},\bar{\rho}}(l,j)$ is the Fourier coefficient of $P_{n_2,\alpha_2,\bar{\chi},\bar{\rho}}(\tau)$. We have two cases: $n_1-\kappa_{\alpha_1}>0$ and $n_1-\kappa_{\alpha_1} = 0$.
Since $a_{n_1,\alpha_1,\chi,\rho}(l,j)$ is the Fourier coefficient of $P_{n_1,\alpha_1,\chi,\rho}(\tau)$,
for each case, we need to prove that the coefficient
\[a_{n_1,\alpha_1,\chi,\rho}(n_2-(\kappa_{\alpha_2}+\kappa'_{\alpha_2}),\alpha_2)\]
is equal to
\begin{equation*}
\begin{cases}
- a_{n_2,\alpha_2,\bar{\chi},\bar{\rho}}(n_1-(\kappa_{\alpha_1}+\kappa'_{\alpha_1}),\alpha_1)\biggl(\frac{-n_2+\kappa'_{\alpha_2}}{n_1-\kappa_{\alpha_1}}\biggr)^{k+1} & \text{if $n_1-\kappa_{\alpha_1}>0$},\\
-\frac{2\pi i^k}{\lambda (k+1)!}\displaystyle\sum_{\gamma = \sm a&b\\c&d\esm\in C^+}c^{-1}\bar{\chi}^{-1}(\gamma)\bar{\rho}(\gamma^{-1})_{\alpha_1,\alpha_2}e^{\frac{2\pi i}{c\lambda}(-n_2+\kappa'_{\alpha_2})a}\biggl(\frac{-2\pi(-n_2+\kappa'_{\alpha_2})}{c\lambda}\biggr)^{k+1} & \text{if $n_1-\kappa_{\alpha_1} = 0$}.
\end{cases}
\end{equation*}
For the first case, we assume that $n_1-\kappa_{\alpha_1}>0$. Then,
 by Theorem \ref{fourierpoincare}, we have
\begin{eqnarray*}
a_{n_1,\alpha_1,\chi,\rho}(n_2-(\kappa_{\alpha_2}+\kappa'_{\alpha_2}),\alpha_2) &=& \frac{2\pi i^{-k-2}}{\lambda}\sum_{\gamma=\sm a&b\\c&d\esm\in C^+}c^{-1}\biggl(\frac{n_2-\kappa'_{\alpha_2}}{n_1-\kappa_{\alpha_1}}\biggr)^{\frac{k+1}2}\chi(\gamma)^{-1}\rho(\gamma^{-1})_{\alpha_2,\alpha_1}\\
&&\times e^{\frac{2\pi i}{c\lambda}((-n_1+\kappa_{\alpha_1})a+(n_2-\kappa'_{\alpha_2})d)}I_{k+1}\biggl(\frac{4\pi}{c\lambda}\sqrt{(n_1-\kappa_{\alpha_1})(n_2-\kappa'_{\alpha_2})}\biggr).
\end{eqnarray*}
Note that $(-I_2)(C^+)^{-1} = C^+$.
If we change $\gamma \mapsto \gamma^{-1}(-I_2)$, then the coefficient $a_{n_1,\alpha_1,\chi,\rho}(n_2-(\kappa_{\alpha_2}+\kappa'_{\alpha_2}),\alpha_2)$ is equal to
\begin{eqnarray} \label{coefficient1}
&&(-1)^{k+2}\frac{2\pi i^{-k-2}}{\lambda}\sum_{\gamma=\sm a&b\\c&d\esm\in C^+}c^{-1}\biggl(\frac{n_2-\kappa'_{\alpha_2}}{n_1-\kappa_{\alpha_1}}\biggr)^{\frac{k+1}2}\bar{\chi}(\gamma)^{-1}\bar{\rho}(\gamma^{-1})_{\alpha_1,\alpha
_2}\\
\nonumber &&\times e^{\frac{2\pi i}{c\lambda}((-n_1+\kappa_{\alpha_1})(-d)+(n_2-\kappa'_{\alpha_2})(-a))}I_{k+1}\biggl(\frac{4\pi}{c\lambda}\sqrt{(n_1-\kappa_{\alpha_1})(n_2-\kappa'_{\alpha_2})}\biggr).
\end{eqnarray}
In the computation, we used that $\chi$ satisfies the non-triviality condition and that $\rho$ is a unitary representation.
Then, by definition,  (\ref{coefficient1}) is the same as
\[- a_{n_2,\alpha_2,\bar{\chi},\bar{\rho}}(n_1-(\kappa_{\alpha_1}+\kappa'_{\alpha_1}),\alpha_1)\biggl(\frac{-n_2+\kappa'_{\alpha_2}}{n_1-\kappa_{\alpha_1}}\biggr)^{k+1}.\]
If $n_1-\kappa_{\alpha_1} = 0$, then by Theorem \ref{fourierpoincare} we have
\begin{eqnarray*}
a_{n_1,\alpha_1,\chi,\rho}(n_2-(\kappa_{\alpha_2}+\kappa'_{\alpha_2}),\alpha_2) &=& \frac{(-2\pi i)^{k+2}}{\Gamma(k+2)\lambda^{k+2}}\sum_{\gamma=\sm a&b\\c&d\esm\in C^+}c^{-{k+2}}(n_2-\kappa'_{\alpha_2})^{k+1}\\
&\times&\chi(\gamma)^{-1}\rho(\gamma^{-1})_{\alpha_2,\alpha_1}e^{\frac{2\pi i}{c\lambda}(n_2-\kappa'_{\alpha_2})d}.
\end{eqnarray*}
Similarly as in the case of $n_1-\kappa_{\alpha_1}>0$ , we change $\gamma$ into $\gamma^{-1}(-I_2)$. Then we obtain that the coefficient $a_{n_1,\alpha_1,\chi,\rho}(n_2-(\kappa_{\alpha_2}+\kappa'_{\alpha_2}),\alpha_2)$ is equal to
\begin{eqnarray*}
\frac{(-2\pi i)^{k+2}}{\Gamma(k+2)\lambda^{k+2}}\sum_{\gamma=\sm a&b\\c&d\esm\in C^+}c^{-{k+2}}(n_2-\kappa'_{\alpha_2})^{k+1}\bar{\chi}(\gamma)^{-1}\bar{\rho}(\gamma^{-1})_{\alpha_1,\alpha_2}e^{\frac{2\pi i}{c\lambda}(n_2-\kappa'_{\alpha_2})(-a)}.
\end{eqnarray*}
Thus, two families $\{P_{n_1,\alpha_1,\chi,\rho}\}$ and $\{G_{n_2,\alpha_2,\bar{\chi},\bar{\rho}}\}$ give a vector-valued grid of weight $k+2$, character $\chi$ and type $\rho$ on $\Gamma$.

Now we prove the uniqueness of a grid. Suppose that there is another grid consisting of two families $\{f'_{n_1,\alpha_1,\chi,\rho}\}$ and $\{G'_{n_2,\alpha_2,\bar{\chi},\bar{\rho}}\}$. Then $G^+_{n_2,\alpha_2,\bar{\chi},\bar{\rho}}(\tau)$ and $G'^+_{n_2,\alpha_2,\bar{\chi},\bar{\rho}}(\tau)$ have the same principal part. Then $G^+_{n_2,\alpha_2,\bar{\chi},\bar{\rho}}(\tau) - G'^+_{n_2,\alpha_2,\bar{\chi},\bar{\rho}}(\tau)$ has no principal part. We consider the pairing defined in \cite{BF} as
\[\{F,G\} = (F,\xi_{-k}(G)),\]
where $F(\tau)\in S_{k+2,\chi,\rho}(\Gamma)$, $G(\tau)\in H^*_{-k,\bar{\chi},\bar{\rho}}(\Gamma)$ and $(\ ,\ )$ is a  Petersson inner product.
Note that $<F,\bar{G}>d\tau$ is a $\Gamma$-invariant $1$-form on $\HH$, where the pairing $<F, \bar{G}>$ is given by
 $\sum_{j=1}^p F_j(\tau) G_j(\tau)$.
 Since $d(<F,\bar{G}>d\tau) = -<F,\overline{L_{-k}G}>\frac{dudv}{v^2}$, we can apply the Stokes' theorem to get
\begin{equation} \label{stokesintegral}
\int_{\mc{F}(\epsilon)} <F(\tau),\overline{L_{-k}G}(\tau)>\frac{dudv}{v^2} = -\int_{\partial\mc{F}(\epsilon)} <F(\tau),\bar{G}(\tau)>d\tau,
\end{equation}
 where $\mc{F}(\epsilon)$ is the punctured domain for $\Gamma$ defined as follows (for more details, see \cite[section 3]{Cho2}).
 Let $\mc{C}$ be the set of parabolic cusps of $\Gamma$. We define an $\epsilon$-disk $B(q, \epsilon)$ at $q\in \mc{C}$ by
\[B(q,\epsilon) := \left\{\tau\in \mc{F}|\ \Im(\gamma_q^{-1} \tau)> \frac1\epsilon\right\},\]
where $\mc{F}$ is the fundamental domain for $\Gamma$.
Let $\mc{F}(\epsilon)$ denote a punctured fundamental domain for $\Gamma$ defined by
\[\mc{F}(\epsilon) = \mc{F} - \bigcup_{q\in\mc{C}} B(q,\epsilon).\]
 By the same argument as in the proof of \cite[Proposition 3.5]{BF} (or \cite[Lemma 3.1]{Cho2}), the above integral (\ref{stokesintegral}) picks the sum of constant terms  $(FG)(\tau)$ at all cusps of $\Gamma$ if we insert the Fourier expansions of $F(\tau)$ and $G(\tau)$. 
Since $\{F,G\} = \lim_{\epsilon\to0} \int_{\mc{F}(\epsilon)} <F(\tau),\overline{L_{-k}G}(\tau)>\frac{dudv}{v^2}$ and $G^+_{n_2,\alpha_2,\bar{\chi},\bar{\rho}}(\tau) - G'^+_{n_2,\alpha_2,\bar{\chi},\bar{\rho}}(\tau)$ has no principal part, we have
\[\{F,G_{n_2,\alpha_2,\bar{\chi},\bar{\rho}}-G'_{n_2,\alpha_2,\bar{\chi},\bar{\rho}}\} = (F,\xi_{-k}(G_{n_2,\alpha_2,\bar{\chi},\bar{\rho}}-G'_{n_2,\alpha_2,\bar{\chi},\bar{\rho}})) = 0\]
for all $F(z)\in S_{k+2,\chi,\rho}(\Gamma)$. This implies that $G^-_{n_2,\alpha_2,\bar{\chi},\bar{\rho}}(\tau) - G'^-_{n_2,\alpha_2,\bar{\chi},\bar{\rho}}(\tau)$ should be zero. Then $G_{n_2,\alpha_2,\bar{\chi},\bar{\rho}}(\tau) - G'_{n_2,\alpha_2,\bar{\chi},\bar{\rho}}(\tau)$ is a cusp form of negative weight and hence it also should be zero. So we showed that $\{G_{n_2,\alpha_2,\bar{\chi},\bar{\rho}}\}$ is unique. Due to the coefficient relation following from the Zagier duality, $\{f'_{n_1,\alpha_1,\chi,\rho}\}$ is completely determined by $\{G'_{n_2,\alpha_2,\bar{\chi},\bar{\rho}}\}$ and hence we see that a vector-valued grid of weight $k+2$, character $\chi$ and type $\rho$ on $\Gamma$ is unique.

Finally, we need to compute the image of $G_{n_2,\alpha_2,\bar{\chi},\bar{\rho}}(\tau)$ under the differential operators $D^{k+1}$. For the differential operator $D^{k+1}$, by the definition of $G_{n_2,\alpha_2,\bar{\chi},\bar{\rho}}(\tau)$, we have
\begin{eqnarray*}
(D^{k+1}G_{n_2,\alpha_2,\bar{\chi},\bar{\rho}})(\tau) &=&
\biggl(\frac{-n_2+\kappa'_{\alpha_2}}{\lambda}\biggr)^{k+1}(D^{k+1}\mc{E}^H_{P_{n_2, \alpha_2, \bar{\chi},\bar{\rho}}})(\tau) = \biggl(\frac{-n_2+\kappa'_{\alpha_2}}{\lambda}\biggr)^{k+1} P_{n_2, \alpha_2, \bar{\chi},\bar{\rho}}(\tau)
\end{eqnarray*}
because $(D^{k+1}\mc{E}^H_{f_{n_2, \alpha_2, \bar{\chi},\bar{\rho}}})(\tau)
= P_{n_2, \alpha_2, \bar{\chi},\bar{\rho}}(\tau)$.
This completes the proof.
\end{proof}

For the proof of
Corollary \ref{main2}, we need the following lemma about the computation of the Petersson inner product. The proof follows along lines familiar from the classical case.

\begin{lem} \label{petersson} Let $P_{n,\alpha,\chi,\rho}(\tau),\ g(\tau) \in S_{k+2,\chi,\rho}(\Gamma)$. Suppose that $g(\tau)$ has the Fourier expansion at $i\infty$ of the form
\[ \sum_{j=1}^p \sum_{l+\kappa_j>0} c(l,j)e^{2\pi i(l+\kappa_j)\tau/\lambda}\mbf{e}_j.\]
Then
\[(g,P_{n,\alpha,\chi,\rho}) = \lambda c(-n,\alpha)\biggl(\frac{\lambda}{4\pi(-n+\kappa_\alpha)}\biggr)^{k+1}\Gamma(k+1).\]
\end{lem}

\begin{proof} [\bf Proof of Lemma \ref{petersson}]
 By the definition of Petersson inner product, we have
\begin{eqnarray*}
(g,P_{n,\alpha,\chi,\rho}) &=& \int_{\Gamma\setminus\HH} < g(\tau),P_{n,\alpha,\chi,\rho}(\tau)>v^{k+2}\frac{dudv}{v^2}.
\end{eqnarray*}
Since $P_{n,\alpha,\chi,\rho}(\tau)$ is a Poincar\'e series defined by
\[P_{n,\alpha,\chi,\rho}(\tau)   = \frac12 \sum_{\gamma=\sm a&b\\c&d\esm\in <T>\setminus\Gamma} \frac{e^{2\pi i(-n+\kappa_{\alpha})\gamma\tau/\lambda}}{\chi(\gamma)(c\tau+d)^{k+2}},\]
we can use the unfolding method to compute the integral:
\begin{eqnarray*}
(g,P_{n,\alpha,\chi,\rho}) &=&
 \int_0^\infty \int_0^\lambda <g(\tau), e^{2\pi i(-n+\kappa_\alpha)\tau/\lambda}>v^{k}dudv\\
  &=&   \int_0^\infty \int_0^\lambda  \sum_{l+\kappa_\alpha>0} c(l,\alpha)e^{2\pi i(l+\kappa_\alpha)\tau/\lambda}e^{-2\pi i(-n+\kappa_\alpha)\bar{\tau}/\lambda}v^k dudv.
\end{eqnarray*}
We can  see that, for integer $n$, we have
\begin{equation*}
\int^\lambda_0 e^{2\pi nu/\lambda}du=
\begin{cases}
0 & \text{if $n\neq 0$},\\
\lambda & \text{if $n=0$}.
\end{cases}
\end{equation*}
Therefore, if we use this computation, we obtain
\begin{eqnarray*}
(g,P_{n,\alpha,\chi,\rho})
&=&\int_0^\infty c(-n,\alpha)e^{-4\pi (-n+\kappa_\alpha)v/\lambda}\lambda v^kdv= \lambda c(-n,\alpha)\biggl(\frac{\lambda}{4\pi (-n+\kappa_\alpha)}\biggr)^{k+1}\Gamma(k+1).
\end{eqnarray*}
This is the desired result.
\end{proof}

Now we are ready to prove Corollary \ref{main2}.

\begin{proof} [\bf Proof of Corollary \ref{main2}]
Note that by Theorem \ref{main1} we have the following relation
\[(\xi_{-k}G_{n,\alpha,\bar{\chi},\bar{\rho}})(\tau) = \frac{(-4\pi)^{k+1}}{k!}\biggl(\frac{-n+\kappa'_\alpha}{\lambda}\biggr)^{k+1}P_{n',\alpha,\chi,\rho}(\tau).\]
Then the proof comes  from the computations of the following Petersson inner product using Lemma \ref{petersson}:
\begin{equation} \label{compare}
(P_{n'_2,\alpha_2, \chi,\rho}, P_{\tilde{n}'_2,\tilde{\alpha}_2, \chi,\rho}).
\end{equation}
Suppose that $P_{n'_2, \alpha_2,\chi,\rho}(\tau)\in S_{k+2, \chi,\rho}(\Gamma)$ has the Fourier expansion of the form
\[P_{n'_2,\alpha_2, \chi,\rho}(\tau) = \sum_{j=1}^p \sum_{l+\kappa_j>0} c_{n'_2,\alpha_2, \chi,\rho}(l, j)e^{2\pi i(l+\kappa_j)\tau/\lambda}\mbf{e}_j.\]
By the same way, we define $c_{\tilde{n}'_2, \tilde{\alpha}_2, \chi,\rho}(l,j)$ is the $l$-th Fourier coefficient of the $j$-th component of $P_{\tilde{n}'_2,\tilde{\alpha}_2,\chi,\rho}(\tau)$.
By Lemma \ref{petersson}, the Petersson inner product (\ref{compare}) equals to
\[\lambda c_{n'_2, \alpha_2, \chi,\rho}(-\tilde{n}'_2, \tilde{\alpha}_2)\biggl(\frac{\lambda}{4\pi(-\tilde{n}'_2+\kappa_{\tilde{\alpha}_2})}\biggr)^{k+1}\Gamma(k+1).\]
Since the Petersson inner product is hermitian,  we see that the Petersson inner product (\ref{compare}) equals to
\[\overline{(P_{\tilde{n}'_2,\tilde{\alpha}_2, \chi,\rho},P_{n'_2,\alpha_2, \chi,\rho} )} =
\lambda \overline{c_{\tilde{n}'_2, \tilde{\alpha}_2, \chi,\rho}(-n'_2, \alpha_2)}\biggl(\frac{\lambda}{4\pi(-n'_2+\kappa_{\alpha_2})}\biggr)^{k+1}\Gamma(k+1).\]
Now we use the fact that $(\xi_{-k}G_{n_2,\alpha_2,\bar{\chi},\bar{\rho}})(\tau)$ is the same as
\begin{equation} \label{imagexi1}
 \frac{(-4\pi)^{k+1}}{\Gamma(k+1)}\biggl(\frac{-n_2+\kappa'_{\alpha_2}}{\lambda}\biggr)^{k+1}P_{n'_2,\alpha_2,\chi,\rho}(\tau).
\end{equation}
The direct computation shows that
\[\xi_{-k}\biggl(H(2\pi(l+\kappa'_j)v/\lambda)e^{2\pi i(l+\kappa'_j)u/\lambda}\biggr) = -e^{-2\pi i(l+\kappa'_j)\tau/\lambda}\biggl(\frac{-4\pi(l+\kappa'_j)}{\lambda}\biggr)^{k+1}.\]
If we use this computation, we see that the Fourier expansion of $(\xi_{-k}G_{n_2,\alpha_2,\bar{\chi},\bar{\rho}})(\tau)$ at $i\infty$ is given in terms of $b^-_{n_2,\alpha_2,\bar{\chi},\bar{\rho}}(l,j)$ as follows:
\begin{eqnarray} \label{imagexi2}
-\sum_{j=1}^p \sum_{l+\kappa'_j<0} \overline{b^-_{n_2,\alpha_2,\bar{\chi},\bar{\rho}}(l,j)}\biggl(\frac{-4\pi(l+\kappa'_j)}{\lambda}\biggr)^{k+1}e^{-2\pi i(l+\kappa'_j)\tau/\lambda}\mbf{e}_j,
\end{eqnarray}
where $b^-_{n_2,\alpha_2}(l,j)$ is the Fourier coefficient of $G^-_{n_2,\alpha_2,\bar{\chi},\bar{\rho}}(\tau)$.
We also have the same kind of equations as in (\ref{imagexi1}) and (\ref{imagexi2}) in the case of $(\xi_{-k}G_{\tilde{n}_2,\tilde{\alpha}_2,\bar{\chi},\bar{\rho}})(\tau)$.
Note that by (\ref{imagexi1}) and (\ref{imagexi2}) we have
\begin{equation} \label{b1}
b^-_{n_2,\alpha_2,\bar{\chi},\bar{\rho}}(-\tilde{n}_2,\tilde{\alpha_2}) = -\overline{a_{n'_2,\alpha_2, \chi,\rho}(\tilde{n}_2-\delta_{\kappa_{\tilde{\alpha}_2},0},\tilde{\alpha_2})}\frac{(-4\pi)^{k+1}}{\Gamma(k+1)}\biggl(\frac{4\pi(\tilde{n}_2-\delta_{\kappa_{\tilde{\alpha}_2},0}+\kappa_{\tilde{\alpha_2}})}  {-n_2+\kappa'_{\alpha_2}}\biggr)^{-(k+1)}.
\end{equation}
By the computations of the Petersson inner product $<P_{n'_2,\alpha_2, \chi,\rho}, P_{\tilde{n}'_2,\tilde{\alpha}_2, \chi,\rho}>$, we see that
\[a_{n'_2, \alpha_2, \chi,\rho}(-\tilde{n}'_2, \tilde{\alpha}_2)\biggl(\frac{1}{-\tilde{n}'_2+\kappa_{\tilde{\alpha}_2}}\biggr)^{k+1} = \overline{a_{\tilde{n}'_2, \tilde{\alpha}_2, \chi,\rho}(-n'_2, \alpha_2)}\biggl(\frac{1}{-n'_2+\kappa_{\alpha_2}}\biggr)^{k+1}.\]
Therefore, (\ref{b1}) is equal to
\[-a_{\tilde{n}'_2,\tilde{\alpha}_2, \chi,\rho}(n_2-\delta_{\kappa_{\alpha_2,0}},\alpha_2)\frac{(-4\pi)^{k+1}}{\Gamma(k+1)}\biggl(\frac{4\pi(n_2+\delta_{\kappa_{\alpha_2},0}+\kappa_{\alpha_2})}  {-n_2+\kappa'_{\alpha_2}}\biggr)^{-(k+1)}.\]
If we again use (\ref{imagexi1}) and (\ref{imagexi2}) for $(\xi_{-k}G_{\tilde{n}_2,\tilde{\alpha}_2,\bar{\chi},\bar{\rho}})(\tau)$, then we see that (\ref{b1}) is the same as
\[\overline{b^-_{\tilde{n}_2,\tilde{\alpha}_2,\bar{\chi},\bar{\rho}}(-n_2,\alpha_2)}\biggl(\frac{-n_2+\kappa'_{\alpha_2}}{-\tilde{n}_2+\kappa'_{\tilde{\alpha}_2}}\biggr)^{k+1},\]
which completes the proof.
\end{proof}

Next, we prove Theorem \ref{main3}. This result follows from the relation coming from the supplementary function theory. Before we prove Theorem \ref{main3}, we state and prove the following lemma.

\begin{lem} \label{periodspace}
Let $\Gamma$ be a $H$-group and $\gamma_1,\cdots, \gamma_t$ be its generators of $\Gamma$. Let $P_k^t$ be the complex vector space of vector-valued polynomials $\sum_{j=1}^t p_j(\tau)\mbf{e}_j$, where the degree of each $p_j(\tau)$, $1\leq j\leq t$, is at most $k$. We define a subspace $\mc{P}_{-k,\chi}(\Gamma)$ as the vector space generated by polynomials given as
\[\sum_{j=1}^t r(f,\gamma_j;\tau)\mbf{e}_j\]
for $f(\tau)\in S_{k+2,\chi}(\Gamma)$. Then the map $\phi: S_{k+2,\chi}(\Gamma) \to \mc{P}_{-k,\chi}(\Gamma)$ given by
\[\phi(f) = \sum_{j=1}^t r(f,\gamma_j;\tau)\mbf{e}_j\]
is an isomorphism.
\end{lem}

\begin{proof} [\bf Proof of Lemma \ref{periodspace}]
Since $\Gamma$ is an $H$-group, it is finitely generated. Therefore, $t$ is a well-defined finite number. By the definition of $\mc{P}_{-k,\chi}(\Gamma)$, the map $\phi$ should be surjective. We only need to prove the injectivity of $\phi$. Suppose that $\phi(f) = 0$ for $f\in S_{k+2,\chi}(\Gamma)$. Then, for each $1\leq j\leq t$, we have
\[r(f,\gamma_j;\tau) = 0.\]
Since $\{\gamma_1,\cdots, \gamma_t\}$ is a generating set of $\Gamma$, this implies that
\[r(f,\gamma;\tau) = c_{k+2}(\mc{E}_f - \mc{E}_f|_{-k,\chi}\gamma)(\tau) =0\]
for all $\gamma\in\Gamma$. Therefore, $\mc{E}_f(\tau)$ is a cusp form of negative weight $-k$ and hence it should be zero. Then we have
\[f(\tau) = D^{k+1}(\mc{E}_f)(\tau) = 0.\]
This proves that the kernel of $\phi$ is trivial. Therefore, $\phi$ is injective and the proof is completed.
\end{proof}

\begin{proof} [\bf Proof of Theorem \ref{main3}]
We define a pairing $\{\ ,\ \}$ for the space $\mc{P}_{-k,\chi}(\Gamma)$ defined in Lemma \ref{periodspace} using the Petersson inner product as follows:
\[\{\phi(f),\ \phi(g)\} := (f,g).\]
Since $\phi$ is an isomorphism by Lemma \ref{periodspace}, this pairing is well-defined. Then this pairing can be extended to the whole space $P^t_k$. We also use the same notation $\{\ ,\ \}$.
Suppose that $\biggl\{\gamma_j =\sm a_j&b_j\\ c_j& d_j\esm \big|\ 1\leq j\leq t\biggr\}$ is a generating set of $\Gamma$. If we use the change of variable $\tau\mapsto \tau+\frac{d_j}{c_j}$ for each $j$-th component, then any element in $P^t_k$ can be written as
\[\sum_{j=1}^t \sum_{i=0}^k a(i,j)\biggl(\tau+\frac{d_j}{c_j}\biggr)^i\mbf{e}_j\]
for some constants $a(i,j)\in \CC$.
Then this pairing can be written as
\begin{equation} \label{pairing}
\biggl\{\sum_{j=1}^t \sum_{i=1}^k a(i,j)\biggl(\tau+\frac{d_j}{c_j}\biggr)^i\mbf{e}_j , \sum_{j=1}^t \sum_{i=1}^k b(i,j)\biggl(\tau+\frac{d_j}{c_j}\biggr)^i\mbf{e}_j\biggr\} = \sum_{1\leq i,j\leq t} \sum_{0\leq\alpha,\beta\leq k}B_{\alpha,\beta}(i,j)a(\alpha,i)\overline{b(\beta,j)},
\end{equation}
where $B_{\alpha,\beta}(i,j)$ is a complex constant.
Now we consider the Petersson inner product
$(P_{n_1,\chi},P_{n_2,\chi})$.
If we use Lemma \ref{petersson}, then we obtain
\[(P_{n_1,\chi},P_{n_2,\chi}) = \lambda c_{n_1,\chi}(-n_2)\biggl(\frac{\lambda}{4\pi(-n_2+\kappa)}\biggr)^{k+1}\Gamma(k+1),\]
where $c_{n_1,\chi}(-n_2)$ is the Fourier coefficient of $P_{n_1,\chi}(\tau)$.
On the other hand, this Petersson inner product can be computed using the pairing (\ref{pairing}). By the definition of the pairing,
\begin{eqnarray} \label{pairing2}
(P_{n_1,\chi},P_{n_2,\chi}) &=& \{\phi(P_{n_1,\chi}),\ \phi(P_{n_2,\chi})\}= \biggl\{ \sum_{j=1}^t r(P_{n_1,\chi},\gamma_j;\tau)\mbf{e}_j,\ \sum_{j=1}^t r(P_{n_2,\chi},\gamma_j;\tau)\mbf{e}_j\biggr\}.
\end{eqnarray}
Note that by Theorem \ref{suppleperiod} we have
\begin{eqnarray*}
r(P_{n_l,\chi},\gamma;\tau) &=& r^H(P_{n_l,\chi},\gamma;\tau)= [r^H(P^*_{n_l,\chi},\gamma;\bar{\tau})]^-
\end{eqnarray*}
for $l=1,\ 2$.  Therefore, we obtain
\begin{eqnarray*}
\biggl\{ \sum_{j=1}^t r(P_{n_1,\chi},\gamma_j;\tau)\mbf{e}_j,\ \sum_{j=1}^t r(P_{n_2,\chi},\gamma_j;\tau)\mbf{e}_j\biggr\} = \biggl\{ \sum_{j=1}^t [r^H(P^*_{n_1,\chi},\gamma;\bar{\tau})]^-\mbf{e}_j,\ \sum_{j=1}^t [r^H(P^*_{n_2,\chi},\gamma;\bar{\tau})]^-\mbf{e}_j\biggr\}.
\end{eqnarray*}
We know that the period polynomial $r^H(P^*_{n_l,\chi},\gamma;\tau)$ can be written in terms of twisted $L$-functions of $P^*_{n_l,\chi}(\tau)$ by Theorem \ref{Lformula}. So we have
\begin{eqnarray*}
r^H(P^*_{n_l,\chi},\gamma_j,\tau) &=&  r(P^*_{n_l,\chi},\gamma_j;\tau) + \delta_{\kappa,0}c_{k+2}\frac{2\pi i^k}{\lambda(k+1)!}\sum_{\gamma=\sm a&b\\c&d\esm\in C^+} c^{-1}\chi(\gamma)e^{\frac{2\pi i}{c\lambda}n_ia}\\
&&\times\biggl(\frac{-2\pi}{c}\biggr)^{k+1}\biggl(1-\chi(\gamma)c_j^k\biggl(\tau+\frac {d_j}{c_j}\biggr)^k\biggr)\\
&=& \sum_{n=0}^{k} i^{1-n}\mat k\\n\emat  \frac{\Gamma(n+1)}{(2\pi)^{n+1}}L(P^*_{n_l,\chi},\zeta_{c_j\lambda}^{-d_j},n+1)\biggl(\tau+\frac {d_j}{c_j}\biggr)^{k-n}\\
&&+\delta_{\kappa,0}c_{k+2}(n_l)^{-(k+1)}\overline{b_{-n_l,\bar{\chi}}(0)}\biggl(1-\chi(\gamma_j)c_j^k\biggl(\tau+\frac{d_j}{c_j}\biggr)^k\biggr)
\end{eqnarray*}
for $l= 1,\ 2$, $1\leq j\leq t$ and $\gamma = \sm a&b\\c&d\esm\in\Gamma$.
Here, $b_{-n_l,\bar{\chi}}(0)$ is the $0$-th coefficient of $G^+_{-n_l,\bar{\chi}}(\tau)$ given by
\[\frac{(-n_l)^{k+1}}{(k+1)!}\sum_{\gamma = \sm a&b\\c&d\esm\in C^+}\biggl(\frac{2\pi}{c\lambda}\biggr)^{k+2}\chi(\gamma)e^{\frac{2\pi i}{c\lambda}n_ia},\]
for $l=1,2$.
Therefore, the pairing in (\ref{pairing2}) can be written in terms of special values of twisted $L$-function of $P^*_{n_1,\chi}(\tau)$ and $P^*_{n_2,\chi}(\tau)$. More precisely, we have
\begin{eqnarray*}
&&\biggl\{ \sum_{j=1}^t [r^H(P^*_{n_1,\chi},\gamma;\bar{\tau})]^-\mbf{e}_j,\ \sum_{j=1}^t [r^H(P^*_{n_2,\chi},\gamma;\bar{\tau})]^-\mbf{e}_j\biggr\}\\
&&= \sum_{1\leq i,j\leq t} \sum_{0\leq\alpha,\beta\leq k}B_{k-\alpha,k-\beta}(i,j) (-i)^{1-\alpha}
\mat k\\ \alpha\emat  \frac{\Gamma(\alpha+1)}{(2\pi)^{\alpha+1}}\overline{L(P^*_{n_1,\chi},\zeta_{c_i\lambda}^{-d_i},\alpha+1)}\\
&&\times i^{1-\beta}\mat k\\ \beta\emat \frac{ \Gamma(\beta+1)}{(2\pi)^{\beta+1}}L(P^*_{n_2,\chi},\zeta_{c_j\lambda}^{-d_j},\beta+1)\\
&&+\delta_{\kappa,0}\overline{c_{k+2}}(n_1)^{-(k+1)}\overline{b_{-n_1,\bar{\chi}}(0)}\sum_{1\leq i,j\leq t}\sum_{0\leq\beta\leq k} (B_{0,k-\beta}(i,j)-\bar{\chi}(\gamma)c_i^kB_{k,k-\beta}(i,j))\\
&&\times i^{1-\beta}\mat k\\ \beta\emat \frac{ \Gamma(\beta+1)}{(2\pi)^{\beta+1}}L(P^*_{n_2,\chi},\zeta_{c_j\lambda}^{-d_j},\beta+1)\\
&&+ \delta_{\kappa,0}c_{k+2}(n_2)^{-(k+1)}b_{-n_2,\bar{\chi}}(0)\sum_{1\leq i,j} \sum_{0\leq\alpha\leq k}(B_{k-\alpha,0}(i,j)-\chi(\gamma)c_j^kB_{k-\alpha,k}(i,j))\\
&&\times  (-i)^{1-\alpha}\mat k\\ \alpha\emat  \frac{\Gamma(\alpha+1)}{(2\pi)^{\alpha+1}}\overline{L(P^*_{n_1,\chi},\zeta_{c_i\lambda}^{-d_i},\alpha+1)}\\
&&+ \delta_{\kappa,0}|c_{k+2}|^2(n_1n_2)^{-(k+1)}\overline{b_{-n_1,\bar{\chi}}(0)}b_{-n_2,\bar{\chi}}(0)\\
&&\times \sum_{1\leq i,j\leq t}(B_{0,0}(i,j)-\bar{\chi}(\gamma)c_i^kB_{k,0}(i,j)-\chi(\gamma)c_j^kB_{0,k}(i,j) + (c_ic_j)^kB_{k,k}(i,j)).
\end{eqnarray*}
To shorten the length of our formula, we define another constants
\begin{eqnarray*}
A_{\alpha,\beta}(i,j) &=& \frac{i^{\alpha-\beta}k!2^{k+1}}{(k-\alpha)!(k-\beta)!\lambda^{k+2}(2\pi)^{\alpha+\beta+k+1}}B_{k-\alpha,k-\beta}(i,j),\\
B_{\alpha}(i) &=& \frac{i^{\alpha-1}c_{k+2}2^{k+1}}{(k-\alpha)!\lambda^{k+2}(2\pi)^{\alpha-k}}\sum_{j=1}^t (B_{k-\alpha,0}(i,j)-\chi(\gamma)c_j^k B_{k-\alpha,k}(i,j)),\\
C_{\beta}(j) &=& \frac{i^{1-\beta}2^{k+1}\overline{c_{k+2}}}{(k-\beta)!\lambda^{k+1}(2\pi)^{\beta-k}}\sum_{i=1}^t(B_{0,k-\beta}(i,j)-\bar{\chi}(\gamma)c_i^kB_{k,k-\beta}(i,j))
\end{eqnarray*}
and
\[D = \frac{|c_{k+2}|^2(4\pi)^{k+1}}{k!\lambda^{k+2}}\sum_{1\leq i,j\leq t}(B_{0,0}(i,j)-\bar{\chi}(\gamma)c_i^kB_{k,0}(i,j)-\chi(\gamma)c_j^kB_{0,k}(i,j) + (c_ic_j)^kB_{k,k}(i,j))\]
for $0\leq \alpha,\beta\leq k$ and $1\leq i,j\leq t$.
If we use these constants and the fact that
\[P^*_{n_i,\chi}(\tau) = P_{n_i',\bar{\chi}}(\tau) = f_{n_i',\bar{\chi}}(\tau),\]
then we get the desired result.
\end{proof}

\subsection*{Acknowledgement}
The authors are grateful to the referee for the careful reading and helpful comments which improved the exposition of this paper.

 

\begin{thebibliography}{99}


\bibitem{Abr} M. Abramowitz and I. Stegun, Pocketbook of Mathematical Functions, Verlag Harri Deutsch (1984).


\bibitem{AKN} T. Asai, M. Kaneko and H. Ninomiya, Zeros of certain modular functions and an application, Comment. Math. Univ. St. Paul. 46 (1997), no. 1, 93-101.






\bibitem{BFK} K. Bringmann, K. Fricke and Z. Kent, Special $L$-values and periods of weakly holomorphic modular forms, Proc. of the Amer. Math. Soc., accepted for publication. 



\bibitem{BKR} K. Bringmann, B. Kane and R. C. Rhoades, Duality and differential operators for harmonic Maass forms, Developments in Mathematics, special volume in honor of Ehrenpreis, 28 (2013), 86-106.

\bibitem{BO} K. Bringmann and K. Ono, Arithmetic properties of coefficients of half-integral weight Maass-Poincar\'e series, Math. Ann., 337 (2007), 591-612.




\bibitem{BF} J. Bruinier and J. Funke, On two geometric theta lifts, Duke Math. J. 125 (1) (2004) 45-90.


\bibitem{BOR} J. Bruinier, K. Ono and R. C. Rhoades, Differential operators for harmonic weak Maass forms and the vanishing of Hecke eigenvalues, Math. Ann. 342 (2008), no. 3, 673-693.


\bibitem{CC} B. Cho and Y. Choie, Zagier duality for harmonic weak Maass forms of integral weight, Proc. Amer. Math. Soc. 139 (2011), no. 3, 787-797.


\bibitem{Cho} D. Choi, A simple proof of Zagier duality for Hilbert modular forms, Proc. Amer. Math. Soc., 134 (2006), 3445-3447.

\bibitem{Cho2} D. Choi, Poincar\'e series and the divisors of modular forms, Proc. Amer. Math. Soc., 138 (2010), no. 10, 3393-3403.


\bibitem{DJ} W. Duke and P. Jenkins, On the zeros and coefficients of certain weakly holomorphic modular forms, Pure and Applied Mathematics Quarterly 4 (2008), no. 4, 1327-1340.








\bibitem{FO} A. Folsom and K. Ono, Duality involving the mock theta function $f(q)$, J. Lond. Math. Soc. (2) 77 (2008), 320-334.


\bibitem{Gim} J. Gim\'enez, Fourier coefficients of vector-valued modular forms of negative weight and Eichler cohomology, Thesis (Ph.D.)-Temple University, 2007, 87 pp.


\bibitem{Gue} P. Guerzhoy, On weak harmonic Maass-modular grids of even integral weights, Math. Res. Lett. 16 (2009), no. 1, 59-65.




\bibitem{HK} S. Y. Husseini and M. I. Knopp, Eichler cohomology and automorphic forms, Illinois J. Math. 15 (1971), 565-577.



\bibitem{Kno} M. I. Knopp, Construction of automorphic forms on $H$-groups and supplementary Fourier series, Trans. Amer. Math. Soc., vol. 103 (1962), 168-188.




\bibitem{KM} M. I. Knopp and G. Mason, Vector-valued modular forms and Poincar\'e series, Illinois J. Math. 48 (2004), no. 4, 1345-1366.


\bibitem{KM2} M. I. Knopp and H. Mawi, Eichler cohomology theorem for automorphic forms of small weights, Proc. Amer. Math. Soc. 138 (2010), no. 2, 395-404.


\bibitem{KR} M. I. Knopp and S. Robins, Easy proofs of Riemann's functional equation for $\zeta(s)$ and of Lipschitz summation, Proc. Amer. Math. Soc. 129 (2001), 1915-1922.

\bibitem{Leh0} J. Lehner, The Fourier coefficients of automorphic forms on horocyclic groups. II, Michigan Math. J. 6 (1959), 173-193.



\bibitem{Leh} J. Lehner, Discontinuous groups and automorphic functions, Math. Surveys, no. 8, American Math. Soc., Providence, R. I., 1964.






\bibitem{Ran} R. A. Rankin, Modular forms and functions, Cambridge University Press, Cambridge, 1977.


\bibitem{Rou} J. Rouse, Zagier duality for the exponents of Borcherds products for Hilbert modular forms, J. Lond. Math. Soc. (2), 73 (2006), 339-354.




\bibitem{Wat} G. N. Watson, A treatise on the theory of Bessel functions, Cambridge Mathematical Library, Cambridge University Press, Cambridge, 1995.




\bibitem{Zag} D. Zagier, Traces of singular moduli, Motives, polylogarithms and Hodge theory, Part I (Irvine, CA, 1998) (2002), Int. Press Lect. Ser., 3, I, Int. Press, Somerville, MA, 211-244.


\bibitem{Zag2} D. Zagier, Ramanujan's mock theta functions and
 their applications, S\'eminaire Bourbaki, Vol. 2007/2008, Ast\'erisque No. 326 (2009), Exp. No. 986, vii-viii, 143-164 (2010).




\bibitem{Zwe} S. Zwegers, Mock theta functions, PhD thesis, Universiteit Utrecht, The Netherlands, 2002.






\end{thebibliography}
\end{document}